\documentclass[a4paper]{article}
\usepackage{ bbm }
\usepackage{amssymb, amsmath, amsthm}
\usepackage{hyperref}
\usepackage[T1]{fontenc}
\usepackage{lmodern} 
\usepackage{mathtools}
\usepackage{mathrsfs}
\usepackage{enumerate}

\title{Minimising Hausdorff Dimension under H\"older Equivalence}
\author{Samuel Colvin}

\usepackage[spacesep,definevectors]{easyvector}

\usepackage[parfill]{parskip}

\makeatletter
\def\thm@space@setup{
  \thm@preskip=\parskip \thm@postskip=0pt
}
\makeatother
\makeatletter
\renewenvironment{proof}[1][\proofname]{\par
  \vspace{-\topsep}
  \pushQED{\qed}
  \normalfont
  \topsep0pt \partopsep0pt 
  \trivlist
  \item[\hskip\labelsep
        \itshape
    #1\@addpunct{.}]\ignorespaces
}{
  \popQED\endtrivlist\@endpefalse
}
\makeatother

\usepackage{graphicx}
\usepackage[all]{xy}
\usepackage{tikz}
\usetikzlibrary{arrows,chains,positioning}

\DeclareMathOperator{\confdim}{Confdim}
\DeclareMathOperator{\holdim}{\textnormal{H\"oldim}}

\DeclareMathOperator{\diam}{diam}

\newcommand{\overbar}[1]{\mkern 1.5mu\overline{\mkern-1.5mu#1\mkern-1.5mu}\mkern 1.5mu}

\makeatletter
\tikzset{join/.code=\tikzset{after node path={
\ifx\tikzchainprevious\pgfutil@empty\else(\tikzchainprevious)
edge[every join]#1(\tikzchaincurrent)\fi}}}
\makeatother
\tikzset{
  >=stealth',every on chain/.append style={join},every join/.style={->}
}
\tikzstyle{labeled}=[execute at begin node=$\scriptstyle,execute at end node=$]

\newcommand{\diff}{\mathop{}\!\mathrm{d}}

\usepackage{fancyhdr}
\newtheorem{thm}{Theorem}[section]
\newtheorem{cor}[thm]{Corollary} 
\newtheorem{prop}[thm]{Proposition} 
\newtheorem{lem}[thm]{Lemma}

\theoremstyle{definition} 
\newtheorem{defn}[thm]{Definition} 
\newtheorem{exmp}[thm]{Example}
\newtheorem{note}[thm]{Note}

\newtheorem{q}[thm]{Question}
\newtheorem{assumption}[thm]{Assumption}

\numberwithin{equation}{thm}

\newcommand{\redit}[1]{{\textit{#1}}}

\newcommand\restr[2]{{
		\left.\kern-\nulldelimiterspace 
		#1 
		\vphantom{\big|} 
		\right|_{#2} 
}}

\newcommand{\Addresses}{{
		\bigskip
		\footnotesize
		
		\textsc{School of Mathematics, University of Bristol, Bristol, BS8 1TW, UK}\par\nopagebreak
		\textit{E-mail address}: \texttt{sc16997@bristol.ac.uk}
}}

\begin{document}

\maketitle

\begin{abstract}
	We study the infimal value of the Hausdorff dimension of spaces that are H\"older equivalent to a given metric space; we call this bi-H\"older-invariant ``H\"older dimension''. This definition and some of our methods are analogous to those used in the study of conformal dimension. 
	
	We prove that H\"older dimension is bounded above by capacity dimension for compact, doubling metric spaces. As a corollary, we obtain that H\"older dimension is equal to topological dimension for compact, locally self-similar metric spaces. In the process, we show that any compact, doubling metric space can be mapped into Hilbert space so that the map is a bi-H\"older homeomorphism onto its image and the Hausdorff dimension of the image is arbitrarily close to the original space's capacity dimension.
	
	We provide examples to illustrate the sharpness of our results. For instance, one example shows H\"older dimension can be strictly greater than topological dimension for non-self-similar spaces, and another shows the H\"older dimension need not be attained.
\end{abstract}

\tableofcontents

\section{Introduction}
Within metric geometry one may encounter the concept of an `Embedding Theorem'. This is usually interpreted to be an answer to the question: ``Can one `faithfully depict' a space, $X$, as a subspace of a `nice' space?''. This type of question has been around for a while and one of the original meanings of it breaks down as `faithfully depict' meaning topological embedding and `nice' meaning Euclidean. The following result is an example of such a theorem, and can be found in~\cite[Theorem VII 5]{hurewicz2015dimension}.

\begin{thm}[Szpilrajn, Eilenberg~{\cite[Theorem 3]{Szpilrajn}}]\label{topembeddingthm}
	Let $X$ be a separable metric space of topological dimension at most $n$, then there exists an embedding $f\colon X\rightarrow I^{2n+1}$, which induces a homeomorphism between $X$ and $f(X)$, such that $f(X)$ has Hausdorff dimension at most $n$. Here $I$ denotes the unit interval $[0,1]\subset \mathbb{R}$.
\end{thm} 

A space has topological dimension at most $n\in\mathbb{N}$ if any open cover can be refined to one where no more than $n+1$ open sets intersect simultaneously, see Definition~\ref{defntopdim}. Whereas, a space has Hausdorff dimension at most $n\in\mathbb{R}$ if it has zero $q$-dimensional Hausdorff measure for any $q>n$, see~\cite[Definition 4.8.]{Mattila}. In both cases, a space has dimension equal to $n$ if $n$ is infimal among all values, $m$, such that the space has dimension at most $m$.

A useful property to note is the following theorem which can be found in~\cite[Theorem VII 2]{hurewicz2015dimension}.
\begin{thm}[Szpilrajn~{\cite[Theorem 2]{Szpilrajn}}]\label{topboundshausfrombelow}
	For any space, $X$, the topological dimension of $X$ is a lower bound for Hausdorff dimension of $X$.
\end{thm}
\begin{exmp}
	The standard $1/3$-Cantor set has topological dimension $0$ and Hausdorff dimension $\log(2)/\log(3)$.
\end{exmp}
 Theorem~\ref{topembeddingthm} is still an interesting statement without the added dimension control on the image, but we include this strengthening because it transitions us toward a second key motivation for our result.
 
A powerful tool in mathematics is the notion of `invariants'; given an equivalence, if one has a property that is invariant under this equivalence, then one can use it to try to distinguish objects as being not equivalent.
\begin{exmp}
	Topological dimension is an invariant of topological equivalence. The standard $1/3$-Cantor set, $C$, has topological dimension $0$, whereas the unit interval $I=[0,1]\subset\mathbb{R}$ has topological dimension $1$, therefore there cannot exist a homeomorphism between $C$ and $I$.
\end{exmp}
 Now, Hausdorff dimension is not an invariant of many important equivalences, for instance topological equivalence, but one can easily derive an invariant from it by associating to a space the infimum of the Hausdorff dimensions of all equivalent spaces. If we consider topological equivalence of separable metric spaces, by combining Theorem~\ref{topembeddingthm} with Theorem~\ref{topboundshausfrombelow}, we obtain that the infimum of the Hausdorff dimensions of spaces homeomorphic to a space $X$ is just the topological dimension for $X$. In this situation we don't obtain a new invariant, we just rephrase an old one. Similarly, Hausdorff dimension is an invariant of isometric and bi-Lipschitz equivalence, so applying this process to either of these equivalences returns Hausdorff dimension. However, we can obtain new, interesting invariants by considering intermediate types of equivalence; for example, if we consider quasi-symmetric equivalence of metric spaces, then we get an invariant called conformal dimension, defined by Pansu in~\cite{pansuconformaldim}, which is fundamentally different from topological and Hausdorff dimension. It often lies strictly between the two. It was originally used to study rank one symmetric spaces, and is now also used to study boundaries of hyperbolic groups, and other fractal metric spaces, see~\cite{confdim}.  
 
 In this paper, we concern ourselves with a form of metric control known as H\"older equivalence. H\"older continuity is prevalent in functional analysis, analysis on metric spaces, partial differential equations, and many other areas of mathematics. Furthermore, the kinds of metric spaces we will be interested in, like boundaries of hyperbolic groups, have a self-similarity property which behaves nicely under this kind of metric distortion.
 \begin{defn}
 	A homeomorphism $f\colon X\rightarrow Y$ between metric spaces $X$ and $Y$ is called \redit{bi-H\"older} if there exist constants $\lambda, \alpha,\beta>0$ such that for any $x,y\in X$
 	\[
 	\frac{1}{\lambda}d(x,y)^{\alpha}\leq d(f(x),f(y))\leq \lambda d(x,y)^{\beta}.
 	\]
 	If such a homeomorphism exists, then $X$ and $Y$ are said to be \redit{H\"older equivalent}.
 \end{defn}

Immediately from the definitions, one obtains the bounds
\begin{equation}\label{equationtrivialdimensionboundsunderbiholder}
\frac{\dim_{H}(X)}{\alpha} \leq \dim_{H}(Y) \leq \frac{\dim_{H}(X)}{\beta},
\end{equation}
but the question of exactly how small $\dim_{H}(Y)$ can be does not seem to have been systematically considered in the literature. Therefore, we apply the above process, to introduce the following invariant of H\"older equivalence.
\begin{defn}
	Let $X$ be a metric space. Define the \redit{H\"older dimension} of $X$, denoted \redit{$\holdim(X)$}, by
	\[
	\holdim(X)\coloneqq \inf\{\dim_{H}(Y)\mid Y \text{ H\"older equivalent to }X\},
	\]
	where $\dim_{H}(Y)$ denotes the Hausdorff dimension of $Y$.
\end{defn}
 H\"older equivalence is a topological equivalence so preserves topological dimension and, therefore, all $Y$ that are  H\"older equivalent to $X$ have topological dimension equal to that of $X$. Combined with Theorem~\ref{topboundshausfrombelow}, we have a trivial lower bound 
\begin{equation}\label{topdimlowerboundforholdim}
\dim(X)\leq \holdim(X),
\end{equation}
where $\dim(X)$ denotes the topological dimension of $X$.

Further, for bounded, uniformly perfect metric spaces, quasi-symmetric equivalences are also H\"older equivalences, see~\cite[Corollary 11.5]{heinonen2001lectures}, hence
\[
\dim(X)\leq \holdim(X)\leq \confdim(X)
\]
for such spaces. For the reader unacquainted with uniform perfectness, it is also sufficient for the metric space to be bounded and connected as uniform perfectness applies constraints on the wildness of gaps in the space, but if there are no gaps, then uniform perfectness is unnecessary. 

Our initial motivation was to study how the Hausdorff dimension of boundaries of hyperbolic groups can vary under H\"older equivalences, prompted by a question posed by Ilya Kapovich in June 2017.
However, our methods apply to much more general metric spaces. One such example is the following corollary on the H\"older dimension of locally self-similar metric spaces.
 These are spaces where small scales are similar to the whole space in a uniform way (see Definition~\ref{defnlocselfsim}). As one might expect, this property is exhibited by classical fractals like the Koch snowflake curve, the Sierpinski carpet, and the $1/3$-Cantor set, but `locally self-similar' is more general than this kind of rigid self-similarity. Indeed, this behaviour is exhibited by boundaries of hyperbolic groups, and the invariant sets of self-similar iterated function systems.

\begin{cor}\label{cormain}
	Let $X$ be a compact, locally self-similar metric space, then $X$ has H\"older dimension equal to its topological dimension.
\end{cor}

Although this means that H\"older dimension is not a new invariant in the context of locally self-similar spaces, Corollary~\ref{cormain} can be also be viewed as a strengthening of the classical Theorem~\ref{topembeddingthm}. In this form, Corollary~\ref{cormain} is the `invariant' phrasing of our work. The following, more general, statement can be thought of as the `embedding theorem' version of our work and is our main theorem.

\begin{thm}\label{thmmain}
	Let $X$ be a compact, $N$-doubling metric space which has capacity dimension $n$ with coefficient $\sigma$, then, for any $q>n$, there exist constants $\mu=\mu(n,q,\sigma,N)>0$, $\alpha=\alpha(n,q,\sigma,N)\geq 1$, and $0<\beta =\beta(n,q,\sigma,N)\leq 1$, and a map $f\colon X \rightarrow \ell^{2}$ such that, for any $x,y\in X$,
	\[
	\frac{1}{\mu\diam(X)^{\alpha}}d(x,y)^{\alpha}\leq d(f(x),f(y))\leq \frac{\mu}{\diam(X)^{\beta}} d(x,y)^{\beta},
	\]
	and the image of $f$ has Hausdorff $q$-measure at most $4^{q}$, and therefore Hausdorff dimension at most $q$.
\end{thm}

The space $\ell^{2}$ is the space of square-summable, real-valued sequences with the the norm $\left|(z_{i})_{i\in\mathbb{N}}\right| = \sum_{i\in\mathbb{N}} z_{i}^{2}$.

In comparison to Theorem~\ref{topembeddingthm}, we upgrade the embedding to one with bi-H\"older metric control, but at the expense of the generality of $X$ and the finite-dimensionality of the range.

 We give definitions of the conditions imposed on $X$ in Section~\ref{definitions}. For now, the reader can think of `capacity dimension' as providing controlled open covers of $X$ at small scales, and `doubling' as giving an upper bound on the number of elements in these controlled covers.

The idea of the proof is that these conditions on $X$ allow us to approximate $X$, at a sequence of scales, by open covers with good properties. We then transfer these approximations over into $\ell^{2}$ using maps with simplicial ranges, which inherit strong metric and dimension control from the properties imposed on these covers. Finally, we take the limit of these approximating maps to get an embedding, and check it has the desired metric and dimension control.

Corollary~\ref{cormain} follows from Theorem~\ref{thmmain} because doubling and capacity dimension equal to topological dimension are both consequences of a compact metric space being locally self-similar. As an intermediate corollary, we have the following.

\begin{cor}\label{cor1}
	Let $X$ be a compact, doubling metric space with capacity dimension $n$, then $X$ has H\"older dimension at most $n$.
\end{cor}

There is a subtlety to H\"older dimension not explicitly mentioned in these results. As H\"older dimension is an infimum, there is a question left open: ``Is this infimum attained and in fact a minimum?'' The following examples illustrate that both possibilities can occur.

The example of $[0,1]\subset\mathbb{R}$ shows that H\"older dimension can be attained as $[0,1]$ has topological and Hausdorff dimension both equal to $1$. Indeed, any example where the topological dimension is equal to the Hausdorff dimension will be an example of this case.

The standard $1/3$-Cantor set, $C$, has H\"older dimension $0$. This can be seen by noting that $C$ is a compact, locally self-similar metric space with topological dimension $0$ and using Corollary~\ref{cormain}.
However, $C$ cannot attain its H\"older dimension, by the trivial bounds~\ref{equationtrivialdimensionboundsunderbiholder}. In Section~\ref{sectionCtimesI}, we show this lack of attainment can happen in all higher dimensions by giving, for any $n\in\mathbb{N}$, an example of a compact, locally self-similar metric space of topological dimension $n$ which does not attain its H\"older dimension under any H\"older equivalence.

\begin{thm}\label{thmCtimesI}
	Let $n\in\mathbb{N}$, $I^{n}=[0,1]^{n}$ be the unit hyper-cube in $\mathbb{R}^{n}$, $C$ be the $1/3$-Cantor set, and $X=C\times I^{n}$ their product with the $\ell^{2}$ metric. Let $Y$ be a H\"older equivalent metric space to $X$. Then $Y$ has Hausdorff dimension strictly greater than $n$. In particular, $C\times I^{n}$ has H\"older dimension $n$ but no H\"older equivalent space attains $n$ as its Hausdorff dimension.
\end{thm}

Essentially, this is because the family $\{\{x\}\times I^{n} \mid x \in C\}$ sitting inside $C\times I^{n}$ is `spread out' and consists of `big' copies of $I^{n}$. Such a family forces the Hausdorff dimension of $C\times I^{n}$ to be strictly more than $n$. These properties are sufficiently preserved by bi-H\"older maps so that any equivalent space also contains a `spread out' family of `big' copies of $I^{n}$, also forcing equivalent spaces to have Hausdorff dimension strictly greater than $n$. 

We have also constructed a non-self-similar Cantor set to illustrate the necessity of some kind of strengthening of topological dimension. 
\begin{thm}\label{thmdimtopwontdo}
	There exists a compact, doubling metric space with topological dimension $0$ but H\"older dimension $1$.
\end{thm}
This example is described in detail in Section~\ref{sectioncapvstop}. Inspiration for this example comes from~\cite{Hakobyan} in which Hakobyan concludes that there exist Cantor sets of Hausdorff dimension $1$ which are minimal for conformal dimension. The fundamental ideas underlying the example are; in the construction of a Cantor set, if one cuts out progressively smaller gaps in proportion to the scale they are cut from, then the Hausdorff dimension of the resulting Cantor set can be forced to be $1$, and if one takes the ratio of gap-to-scale to grow faster than any power, then this property cannot be broken by passing through a bi-H\"older map. 

Finally, we present an example to illustrate that H\"older dimension can be strictly less than capacity dimension, verifying that the estimate ``at most $n$'' in Corollary~\ref{cor1} is necessary.

\begin{thm}
	There exists a compact, doubling metric space with capacity dimension $1$ but H\"older dimension $0$.
\end{thm}

The idea is that it only takes a countable number of points spaced out poorly to force capacity dimension to increase beyond $0$, but a countable number of points has Hausdorff dimension $0$. This example is explicitly described in Section~\ref{sectionholdimnotalwayscap}.

The reader should note that every example we've given has had integer H\"older dimension. For our compact, locally self-similar examples, this follows from Corollary~\ref{cormain}. However, even in our non-self-similar examples we still have an integer H\"older dimension. This leads one to pose the question: 
\begin{q}
	Can H\"older dimension take a non-integer value? 
\end{q}

Another possible line of further investigation is to try to improve upon our choice of $\ell^{2}$ by embedding, instead, into some finite dimension space, $\mathbb{R}^{N}$, for sufficiently large $N$.

\subsection*{Acknowledgements}
The author is grateful to John Mackay for invaluable guidance.

\section{Capacity dimension and metric spaces}\label{definitions}

A major component of the proof of Theorem~\ref{thmmain} is approximating $X$ by sequentially finer covers. However, not any old haphazard covers will do; we will need them to have quite a bit of structure. In this section, we delve into some background on how to cover metric spaces with open sets, and how concepts like dimension and doubling allow us to take covers with more structure. In particular, we define topological dimension, capacity dimension, and doubling, and provide a lemma which combines capacity dimension and doubling to prove the existence of covers with especially useful properties. We start with the notion of topological dimension. The following definitions are from~\cite[Section 50]{munkrestopology}.

Let $X$ be a topological space.

\begin{defn}
	A collection $\mathcal{A}$ of subsets of the space $X$ is said to have \redit{multiplicity} $m$ if some point $X$ lies in $m$ elements of $\mathcal{A}$, and no point of $X$ lies in more than $m$ elements of $\mathcal{A}$.
\end{defn}


\begin{defn}
	Given a collection $\mathcal{A}$, a collection $\mathcal{B}$ is said to \redit{refine} $\mathcal{A}$, or to be a \redit{refinement} for $\mathcal{A}$, if for each element $B$ of $\mathcal{B}$, there is an element $A$ of $\mathcal{A}$ such that $B\subset A$.
\end{defn}

\begin{defn}
	A collection $\mathcal{A}$ of subsets of a space $X$ is said to \redit{cover} $X$, or to be a \redit{covering} of $X$, if the union of the elements of $\mathcal{A}$ is equal to $X$. It is called a \redit{open covering} of $X$ if its elements are open subsets of $X$.
\end{defn}

These definitions culminate in the following topological notion of dimension, also sometimes also referred to as ``covering dimension''.

\begin{defn}\label{defntopdim}
	The \redit{topological dimension} of $X$ is defined to be the smallest integer $m$ such that for every open covering $\mathcal{A}$ of $X$, there exists an open covering $\mathcal{B}$ that refines $\mathcal{A}$ and has multiplicity at most $m+1$.
\end{defn}

We need better control on the size of elements in these covers so we look to a related notion called capacity dimension, which can be found in~\cite{BuyaloLebadeva}, that strengthens topological dimension in metric spaces by imposing metric constraints.

Let $X$ be a metric space.

\begin{defn}
	The \redit{mesh} of a covering $\mathcal{U}$ is the supremum of the diameters of elements of $\mathcal{U}$.
	\[
	\text{mesh}(\mathcal{U}) \coloneqq \sup\{\diam(U)\mid U\in\mathcal{U}\}.
	\]
\end{defn}

\begin{defn}
	A covering $\mathcal{U}$ is said to be \redit{coloured} if it is the union of $m\geq1$ disjoint families, $\mathcal{U} = \cup_{a\in A} \mathcal{U}^{a}$, $|A|=m$, with the property that, for any $a\in A$, if $U,V\in\mathcal{U}^{a}$ are distinct, then $U\cap V = \emptyset$.  In this case we also say that $\mathcal{U}$ is $m$-coloured.
	
\end{defn}

Note that an $m$-coloured covering $\mathcal{U}$ has multiplicity at most $m$. Indeed, if some members of $\mathcal{U}$ have non-empty intersection, then they must each lie in different families of which there are $m$, and, therefore, at most $m$ can intersect non-trivially.

\begin{defn}
	Let $\mathcal{U}$ be a family of open subsets in a metric space $X$ which cover $A\subset X$. Given $x\in A$, we let
	\[
	\mathcal{L}(\mathcal{U},x) \coloneqq \sup\{d(x,X\setminus U)\mid U\in\mathcal{U}\}
	\]
	be the \redit{Lebesgue number of $\mathcal{U}$ at $x$}, $\mathcal{L}(\mathcal{U}) = \inf_{x\in A} \mathcal{L}(\mathcal{U},x)$ be the \redit{Lebesgue number of the covering $\mathcal{U}$ of $A$}.
\end{defn}
We give the definition of Lebesgue number from~\cite{BuyaloLebadeva} as this is the definition Buyalo and Lebadeva use when giving Definition~\ref{capdimdefn}.

This definition is a little opaque, but, luckily, the reader need only concern themselves with the following key fact about Lebesgue number. 
\begin{lem}
	If $\mathcal{U}$ is a finite cover for $A\subset X$, then, for every $x\in A$, the open ball $B(x,r)$ in $X$ of radius $r\leq \mathcal{L}(\mathcal{U})$ centred at $x$ is contained in some element of the cover $\mathcal{U}$.
\end{lem}

\begin{defn}\label{capdimdefn}
	The \redit{capacity dimension} of a metric space $X$ is the minimal integer $n\geq 0$ with the following property: There is a constant $\sigma^{\prime}\in (0,1)$ such that for every sufficiently small $\delta>0$ there exists an $(n+1)$-coloured open covering $\mathcal{U}^{\prime}$ of $X$ with mesh$(\mathcal{U}^{\prime})\leq \delta$ and $\mathcal{L}(\mathcal{U}^{\prime})\geq \sigma^{\prime} \delta$. We say that $X$ has capacity dimension $n$ with coefficient $\sigma^{\prime}$.
\end{defn}

Buyalo and Lebadeva then proceed to give some conditions for which one can assume that capacity dimension and topological dimension are equal.

\begin{defn}\label{defnlocselfsim}
	A metric space $(X,d)$ is \redit{locally self-similar} if there exists $\lambda\geq 1$ such that for every sufficiently large $R>1$ and every $A\subseteq X$ with $\diam(A)\leq \Lambda_{0}/R$, where $\Lambda_{0} = \min\{1, \diam(X)/\lambda\}$, there is an embedding 
	\[
	f\colon A\rightarrow X
	\]
	such that, for all $z_{1},z_{2}\in A$, 
	\[
	Rd(z_{1},z_{2})/\lambda\leq d(f(z_{1}),f(z_{2}))\leq \lambda Rd(z_{1},z_{2}).
	\]
	In other words, $f$ is a $\lambda$-bi-Lipschitz homeomorphism from $(A,Rd)$, the subspace $A$ with a rescaled metric, to its image in $(X,d)$.
\end{defn}

The following theorem is corollary 1.2 in~\cite{BuyaloLebadeva}.

\begin{thm}\label{BLcor}
	The capacity dimension of every compact, locally self-similar metric space $X$ is finite and coincides with its topological dimension.
\end{thm}

We will, in fact, not need the full strength of local self-similarity for our main theorem, but it is helpful to see here that a space having finite capacity dimension is not a particularly unreasonable assumption.

Finally, we need control on how many elements are in these covers. To this end, we introduce the concept of a doubling metric space which can be found in~\cite[Chapter 10]{heinonen2001lectures}.

\begin{defn}\label{doublingdefn}
	A metric space is \redit{doubling} if there exists a constant $N<\infty$ such that, for any $x\in X$ and $r>0$, any ball $\overbar{B(x,r)}$ can be covered by at most $N$ balls of radius $r/2$. In particular, we say $X$ is $N$-doubling.
\end{defn}

\begin{lem}\label{final sizecontrolledcovers}
	Suppose $X$ is a finite diameter, $N$-doubling metric space of capacity dimension $n$ with coefficient $\sigma^{\prime}$, then there is a constant $\sigma = \sigma^{\prime}/2\in (0,1)$ such that, for every sufficiently small $\delta>0$, there exists an $(n+1)$-coloured open covering $\mathcal{U}$ of $X$ with mesh$(\mathcal{U})\leq \delta$, $\mathcal{L}(\mathcal{U})\geq \sigma \delta$, and 
	\[
	|\mathcal{U}|\leq N^{\log_{2}(2\diam(X)/\sigma\delta)}.
	\]
\end{lem}

\begin{proof}
	Given that $X$ is $N$-doubling, we can cover $X$ by at most $N$ balls of radius $\diam(X)/2$. Each of these balls can be covered by at most $N$ balls of radius $\diam(X)/4$, so $X$ can be covered by at most $N^{2}$ balls of radius $\diam(X)/4$. Continuing inductively, we see that $X$ can be covered by at most $N^{k}$ balls of radius $\diam(X)/2^{k}$ for any $k\in \mathbb{N}$. 
	Let $\delta>0$ be sufficiently small as in Definition~\ref{capdimdefn} of capacity dimension, and fix $k\in \mathbb{N}$ to be the unique positive integer such that
	\[
	\frac{\diam(X)}{2^{k}}< \frac{\sigma^{\prime}\delta}{2} \leq\frac{\diam(X)}{2^{k-1}},
	\]
	equivalently,
	\[
	\frac{\sigma^{\prime}\delta}{4}\leq\frac{\diam(X)}{2^{k}}< \frac{\sigma^{\prime}\delta}{2}.
	\]
	Rearranging, we see $k$ satisfies
	\[
	\log_{2}\left(\frac{2\diam(X)}{\sigma^{\prime}\delta}\right)< k\leq \log_{2}\left(\frac{4\diam(X)}{\sigma^{\prime}\delta}\right).
	\]
	Now, let $\mathcal{U}^{\prime}$ be a cover for $X$ of mesh at most $\delta$ provided by the definition of capacity dimension. In particular, $\mathcal{L}(\mathcal{U}^{\prime})\geq \sigma^{\prime}\delta$, so for any $x\in X$, $B(x,\sigma^{\prime}\delta)\subset U$ for some $U\in\mathcal{U}^{\prime}$. Consider also a cover $\mathcal{D}$ of $X$ by balls of radius $\sigma^{\prime}\delta/2$. By $N$-doubling and the above, we can assume, 
	\[
	\left|\mathcal{D}\right|\leq N^{k} \leq N^{\log_{2}(4\diam(X)/\sigma^{\prime}\delta)}.
	\]
	Say $\mathcal{D} = \{B(x_{j},\sigma^{\prime}\delta/2)\}_{j=1}^{N^{k}}$, then for each $j$, using the Lebesgue number of $\mathcal{U}^{\prime}$ we can pick a $U_{j}\in\mathcal{U}^{\prime}$ such that $B(x_{j},\sigma^{\prime}\delta/2)\subseteq B(x_{j},\sigma^{\prime}\delta)\subseteq U_{j}$. Now, note that $\mathcal{U} \coloneqq\{U_{j}\}_{j=1}^{N^{k}}$ is also an open covering of mesh at most $\delta$, $(n+1)$-coloured, but now with $\left|\mathcal{U}\right|$ at most $N^{k}$. We're almost there, but we might have decreased the Lebesgue number by removing elements from $\mathcal{U}$. To fix this, note that $B(x_{j},\sigma^{\prime}\delta/2)$ is still an open cover for $X$, so, for any $x\in X$ there exists a $j$ such that $x\in B(x_{j},\sigma^{\prime}\delta/2)$. Consequently, $B(x,\sigma^{\prime}\delta/2)\subseteq  B(x_{j},\sigma^{\prime}\delta)$ by the triangle inequality, and therefore $B(x,\sigma^{\prime}\delta/2)\subseteq U_{j}$. Hence, $\mathcal{L}(\mathcal{U})\geq \sigma^{\prime}\delta/2$, and taking $\sigma\coloneqq \sigma^{\prime}/2$, we get the desired result that $\mathcal{U}$ has capacity dimension properties plus
	\[
	\left|\mathcal{U}\right|\leq N^{k}\leq N^{\log_{2}(2\diam(X)/\sigma\delta)}.\qedhere
	\]
\end{proof}

\section{Construction of the approximating maps}\label{constructionsec}

We now begin the proof of Theorem~\ref{thmmain}. We start in Subsection~\ref{subsectionapproximatingXbynicecoversatcontrolledscales} by defining a sequence of scales, approximating $X$ by nice covers at each of these scales, and then building functions, in Subsection~\ref{subsectionconstructingmaps} which translate these approximating covers over into approximations of $X$ in $\ell^{2}$. We will then proceed to, in Section~\ref{sectionpropertiesofthemaps}, show some useful properties of these approximating functions. 

\subsection{Approximating $X$ by nice covers at controlled scales}\label{subsectionapproximatingXbynicecoversatcontrolledscales}

Assume $X$ is a compact, doubling metric space which has capacity dimension $n$. For now, we prove Theorem~\ref{thmmain} assuming $X$ has diameter equal to $1$.

Let $\sigma$ be as in Lemma~\ref{final sizecontrolledcovers}, and let $L=L(n,\sigma)$ be defined by
\begin{equation}\label{defnL}
L\coloneqq\frac{128(n+1)^{2}}{\sigma^{2}}.
\end{equation}
Later, we will show that certain maps are locally Lipschitz, and $L$ will appear in the corresponding Lipschitz constants.

Let $q>n$. We will construct a space which is H\"older equivalent to $X$ and has Hausdorff dimension at most $q$.

\begin{note}\label{Nbig}
	The `$N$' in $N$-doubling, see Definition~\ref{doublingdefn}, is really only an upper bound; if $N^{\prime}\geq N$ and $X$ is $N$-doubling, then $X$ is $N^{\prime}$-doubling too. Similarly, if $\mathcal{U}$ is a cover as in Lemma~\ref{final sizecontrolledcovers}, then for every $N^{\prime}\geq N$
	\[
	|\mathcal{U}| \leq{(N^{\prime})}^{\log_{2}(2\diam(X)/\sigma\delta)},
	\]
	because the exponent is positive considering that $\delta\leq \diam(X)$ and $0<\sigma<1$.
	
	Throughout, it will be convenient to assume $N$ is much larger than other constants. Indeed, as $n,q,\sigma$ are all such that replacing $N$ with a larger value has no effect on them, we can assume $N$ is arbitrarily large with respect to $n,q$, and $\sigma$. The exact value of $N=N(n,q,\sigma)$ is determined by Lemma~\ref{assumptionsonN} later.
\end{note}

Let $\delta_{0} = 1$, $\epsilon_{0} = 1$, and inductively define $\epsilon_{i+1} = \epsilon_{i+1}(n,q,\sigma,N,\delta_{i})$ by 
\begin{equation}\label{epsiloni+1}
\epsilon_{i+1} \coloneqq  \frac{1}{8}\left((8\sqrt{n})^{n}N^{(n+2)\log_{2}(2/\sigma\delta_{i})}\right)^{\frac{-1}{q-n}} 
\end{equation}
and $\delta_{i+1} = \delta_{i+1}(L,\delta_{i},\epsilon_{i},\epsilon_{i+1})$ by
\begin{equation}\label{deltai+1}
\delta_{i+1} \coloneqq \frac{1}{L}\frac{\delta_{i}}{\epsilon_{i}}\epsilon_{i+1},
\end{equation}
for $i\geq 0$.

In Subsection~\ref{subsectionconstructingmaps} we will construct a sequence of maps from $X$ to $\ell^{2}$ dependent on these sequences $(\epsilon_{j})_{j\in\mathbb{N}}$ and $(\delta_{j})_{j\in\mathbb{N}}$.
One should think of $(\epsilon_{j})_{j\in\mathbb{N}}$ and $(\delta_{j})_{j\in\mathbb{N}}$ as partitioning distances in $\ell^{2}$ and $X$, respectively, into different `scales'. The maps will approximate $X$ at scale $\delta_{i}$ by simplicies at scale $\epsilon_{i}$ in $\ell^{2}$. Our exact choices of these $\epsilon_{i}$ and $\delta_{i}$ are made to give us good control on how the maps distort distance. 
For now, we note some properties of these two sequences as they will be integral to the overall proof.

\begin{lem}\label{final epsilondeltarelation}
	The sequences $(\delta_{j})_{j\in\mathbb{N}}$ and $(\epsilon_{j})_{j\in\mathbb{N}}$ satisfy
	\[
	\delta_{i+1} = \left(\frac{1}{L}\right)^{i+1} \epsilon_{i+1},\text{ for all }i\geq 0.
	\]
\end{lem}

\begin{proof}
	Inductively apply the recurrence relation and note you have a telescoping product. Finally, recall that $\epsilon_{0} = \delta_{0} =1$.
\end{proof}

\begin{lem}\label{epsilon decreasing}
	The sequence $(\epsilon_{j})_{j\in\mathbb{N}}$ satisfies 
	\[
	\frac{\epsilon_{i+1}}{\epsilon_{i}}\leq \frac{1}{L},\text{ for all }i\geq 0.
	\]
\end{lem}

In particular, throughout, we will also require $(\epsilon_{j})_{j\in\mathbb{N}}$ to be such that
\begin{equation*}
	\frac{\epsilon_{i+1}}{\epsilon_{i}} \leq 1/(8\sqrt{2(n+1)}) \leq 1/2,
\end{equation*}
which is a consequence of Lemma~\ref{epsilon decreasing} by the definition of $L$ given in~\eqref{defnL}.

\begin{proof}
	If $i\geq 1$ then by substituting out $\delta_{i}$ using Lemma~\ref{final epsilondeltarelation} in the expression of $\epsilon_{i+1}$ given in~\eqref{epsiloni+1}, we see that 
	
	\begin{align}
	\epsilon_{i+1} &= \frac{1}{8}\left((8\sqrt{n})^{n}N^{(n+2)\log_{2}\left(2L^{i}/\sigma\epsilon_{i}\right)}\right)^{\frac{-1}{q-n}},\nonumber\\
	&=\frac{1}{8}\left(\frac{1}{B_{1}N^{B_{2}}N^{iB_{3}}}\right)N^{\frac{n+2}{q-n}\log_{2}(\epsilon_{i})},\nonumber\\
	&=\frac{1}{8}\left(\frac{1}{B_{1}N^{B_{2}}N^{iB_{3}}}\right)\epsilon_{i}^{\frac{n+2}{q-n}\log_{2}(N)},\label{explicitepsiloni+1}
	\end{align}
	where 
	\begin{equation}\label{defnB1}
		B_{1}\coloneqq B_{1}(n,q)\coloneqq (8\sqrt{n})^{\frac{n}{q-n}}>0,
	\end{equation}
	\begin{equation}\label{defnB2}
	B_{2} \coloneqq B_{2}(n,q,\sigma)\coloneqq \frac{n+2}{q-n}\log_{2}\left(\frac{2}{\sigma}\right)>0,
	\end{equation}
	\begin{equation}\label{defnB3}
	B_{3} \coloneqq B_{3}(n,q,L) \coloneqq \frac{n+2}{q-n}\log_{2}(L)>0.
	\end{equation}
	The statements that $B_{2}$ and $B_{3}$ are positive come from $0<\sigma<1$ for $B_{2}$, and $L>1$ for $B_{3}$. 
	
	Similarly, we can simplify the expression for $\epsilon_{1}$ from~\eqref{epsiloni+1} without substituting out $\delta_{0}$ but instead recalling that $\delta_{0} = \epsilon_{0} =1$ to get
	\[
	\epsilon_{1}=\frac{1}{8}\left(\frac{1}{B_{1}N^{B_{2}}}\right)\epsilon_{0}^{\frac{n+2}{q-n}\log_{2}(N)},
	\]
	which extends~\eqref{explicitepsiloni+1} to include the $i=0$ case.
	
	As $B_{3}>0$ and $N\geq 1$, we know that $N^{B_{3}}\geq 1$. Furthermore, by Note~\ref{Nbig}, we can assume
	\[
	\frac{n+2}{q-n}\log_{2}(N)\geq 1
	\]
	and
	\[
	B_{1}N^{B_{2}}N^{iB_{3}}\geq B_{1}N^{B_{2}}\geq L,
	\]
	for all $i\geq 0$.
	Now, as~\eqref{explicitepsiloni+1} holds for all $i\geq 0$ and $\epsilon_{0}=1$, we can observe that, inductively, $\epsilon_{i}\leq 1$ and therefore $\epsilon_{i+1}\leq \frac{1}{L} \epsilon_{i}$.
\end{proof}

By combining Lemma~\ref{epsilon decreasing}, Note~\ref{Nbig}, and the definition of $\delta_{i}$ in~\eqref{deltai+1}, we see that we can assume that $\delta_{i}$ is sufficiently small as to apply Lemma~\ref{final sizecontrolledcovers} with $\delta = \delta_{i}$ for all $i\geq1$.

Let $\mathcal{U}_{i+1}$ be a cover of $X$, for all $i\geq0$, as in Lemma~\ref{final sizecontrolledcovers} with mesh$(\mathcal{U}_{i+1})\leq\delta_{i+1}$. In particular, $\mathcal{L}(\mathcal{U}_{i+1})\geq \sigma \delta_{i+1}$, $\mathcal{U}_{i+1}$ has multiplicity at most $n+1$, and 
\begin{equation}\label{Uicover}
\left|\mathcal{U}_{i+1}\right|\leq N^{\log_{2}(\frac{2}{\sigma\delta_{i+1}})}.
\end{equation}
\begin{note}\label{notecoversareapproximations}
These open sets are collections of points which have distances at most $\delta_{i+1}$ from each other, and, by the Lebesgue number property, any collection of points with distances bounded above by $\sigma\delta_{i+1}$ lies in one of these open sets. Therefore, one could interpret such a covering as approximating $X$ by objects of roughly the scale $\delta_{i+1}$.
\end{note}
\begin{assumption}\label{notenonredundantcovers}
	Later, in the proof of Proposition~\ref{lemlocallylipschitz}, we will assume that these covers, $(\mathcal{U}_{j})_{j\in\mathbb{N}_{\geq1}}$, satisfy a kind of `non-redundancy' property: For any $\mathcal{U}_{i+1}$, $i\geq 0$, and any distinct $U,V\in\mathcal{U}_{i+1}$, we have that $U\nsubseteq V$. This assumption is justified as if there exists a pair of distinct elements $U,V\in\mathcal{U}_{i+1}$ such that $U\subseteq V$, then $\mathcal{U}_{i+1}\setminus\{U\}$ is still an open cover of $X$ with mesh at most $\delta_{i+1}$, Lebesgue number at least $\sigma \delta_{i+1}$, multiplicity at most $n+1$, and $|\mathcal{U}_{i+1}|\leq N^{\log_{2}(2/\sigma\delta_{i+1})}$. Therefore, replacing $\mathcal{U}_{i+1}$ with a cover which has had all of the `redundant' elements removed in this way gives us a new cover with all the same desired properties.
\end{assumption}

\subsection{The construction}\label{subsectionconstructingmaps}

Recall that $\ell^{2}$ is the space of square-summable, real-valued sequences with the the norm $\left|(z_{i})_{i\in\mathbb{N}}\right| = \sum_{i\in\mathbb{N}} z_{i}^{2}$. Let $f_{0}\colon X\rightarrow \ell^{2}$ be the constant zero map; $f_{0}(x)= (0,0,0,\dots)$ for all $x\in X$. We now inductively define a sequence of maps $(f_{j})_{j\in\mathbb{N}}$ which will approximate $X$ in $\ell^{2}$ to progressively finer scales.

	Suppose you have a map $f_{i}\colon X \rightarrow \ell^{2}$ such that the image of $f_{i}$ is contained in $\{(z_{1},\dots,z_{m_{i}},0,0,\dots)\mid z_{i}\in\mathbb{R}\}$ for some $m_{i}$. In other words, $f_{i}(X)$ is contained in a particular finite-dimensional linear subspace of $\ell^{2}$. Note that this condition does indeed hold for $f_{0}$ as $f_{0}(X) = \{(0,0,\dots)\}$ is of this form. 

	Order $\mathcal{U}_{i+1}=\{U_{1},U_{2},\dots U_{\left|\mathcal{U}_{i+1}\right|}\}$ and for each $1\leq k\leq \left|\mathcal{U}_{i+1}\right|$, pick $x_{k}\in U_{k}$. 
Then define,
	\begin{equation}\label{pU}
	p_{k}=f_{i}(x_{k}) + \frac{\epsilon_{i+1}}{2} e_{m_{i}+k},
	\end{equation}
	where $e_{j} = (\delta_{k,j})_{k\in\mathbb{N}}$, $\delta_{k,j} =1 $ if $k=j$ and $0$ otherwise.
	In other words, if $f_{i}(x_{k})$ has the form 
	\[
	(z_{1},z_{2},\dots,z_{m_{i}},0,0,\dots),
	\] 
	then $p_{k}$ has the form 
	\[
	(z_{1},\dots,z_{m_{i}},0,\dots,0,\epsilon_{i+1}/2,0,0,\dots),
	\]
	where the $\epsilon_{i+1}/2$ is in the $(m_{i}+k)$-th co-ordinate. 
	
	We will sometimes write $p_{U}\coloneqq p_{k}$ if $U=U_{k}\in\mathcal{U}_{i+1}$.
	
	Now, define $f_{i+1}\colon X \rightarrow \ell^{2}$ as follows
	\begin{equation}\label{fi+1}
	x\mapsto \frac{\sum_{k=1}^{\left|\mathcal{U}_{i+1}\right|} d(x,X\setminus U_{k})p_{k}}{\sum_{k=1}^{\left|\mathcal{U}_{i+1}\right|} d(x,X\setminus U_{k})}.
	\end{equation}
	
	Note, for any $x\in X$, $f_{i+1}(x)$ is a (finite) linear combination of vectors contained in $\{(z_{1},\dots,z_{m_{i+1}},0,0,\dots)\mid z_{i}\in\mathbb{R}\}$, where $m_{i+1} = m_{i}+\left|\mathcal{U}_{i+1}\right|$, and therefore $f_{i+1}(x)$ is also contained in this set. This justifies that we can indefinitely continue this inductive definition of functions to get an infinite sequence $(f_{j})_{j\in\mathbb{N}}$.
	
	\section{Properties of the approximating maps 
	}\label{sectionpropertiesofthemaps}
	In this section, we show that the approximating maps, $f_{i}$, are locally-Lipschitz, that points that are clearly distinct at scale $\delta_{i}$ are clearly distinct in the image by scale $\epsilon_{i}$, which we refer to as a ``separation'' property, and finally that we have good control of the $q$-measure of the images, $f_{i}(X)$.
	\subsection{Locally Lipschitz}
	The functions $f_{i}$ have been chosen so that they don't stretch distances too far. More precisely,
	
	\begin{prop}\label{lemlocallylipschitz}
		For all $i\geq 0$, $f_{i}$ is locally Lipschitz. In particular, if $d(x,y)<\sigma\delta_{i}$, then $d(f_{i}(x),f_{i}(y))\leq \frac{L}{2}\frac{\epsilon_{i}}{\delta_{i}}d(x,y)$, where $L$ is defined in~\eqref{defnL}.
	\end{prop}

We actually only require the following corollary.

\begin{cor}\label{controlledstretching}
	For all $i\geq0$, 
	\begin{equation}\label{lipshitzcor}
		 d(x,y)\leq \delta_{i+1} \implies d(f_{i}(x),f_{i}(y))\leq \epsilon_{i+1}/2.
	\end{equation}
\end{cor}

\begin{proof}
	The case of $i=0$ trivially holds because $f_{0}(x)=f_{0}(y)$ for all $x,y\in X$. For $i\geq 1$, recall the definition of $\delta_{i+1}$ from~\eqref{deltai+1},
	\[
	\delta_{i+1} = \frac{1}{L}\frac{\delta_{i}}{\epsilon_{i}}\epsilon_{i+1} = \frac{\sigma^{2}}{128(n+1)^{2}}\frac{\delta_{i}}{\epsilon_{i}}\epsilon_{i+1},
	\] 
	and note that $\sigma<1$, $128(n+1)^{2}\geq 1$, and $\epsilon_{i+1}\leq \epsilon_{i}$ from Lemma~\ref{epsilon decreasing}, so $\delta_{i+1}\leq \sigma \delta_{i}$, and therefore by Proposition~\ref{lemlocallylipschitz},
	\[
	d(x,y)\leq \delta_{i+1} \implies d(f_{i}(x),f_{i}(y))\leq \epsilon_{i+1}/2.\qedhere
	\]
\end{proof}

First we need some lemmas. Throughout the following lemmas, we impose conditions so that open sets, $U$, are not the entirety of $X$; this is simply so that $d(x,X\setminus U)$ is well-defined for $x\in X$.
	
	\begin{lem}\label{lemweightfunctionsarelipschitz}
		For $X$ a metric space, $x,y\in X$, and $U\subsetneq X$ an open set, we have
		\[
		|d(y,X\setminus U)-d(x,X\setminus U)| \leq d(x,y).
		\]
	\end{lem}
	\begin{proof}
		For any $z\in X\setminus U$,  by definition and then the triangle inequality, we have
		\[
		d(y,X\setminus U)\leq d(y,z) \leq d(x,y) +d(x,z).
		\]
		Hence,
		\[
		d(x,X\setminus U) = \inf_{z\in X\setminus U}d(x,z) \geq d(y,X\setminus U)- d(x,y).
		\]
		This argument was symmetric in $x$ and $y$, so the lemma holds.
	\end{proof}

\begin{lem}\label{lemisolatedopensetsareokay}
	Suppose $X$ is a metric space, and $\mathcal{U}$ is an open cover of $X$ with mesh at most $0<\delta<\diam(X)$ such that, for any distinct $U,V\in\mathcal{U}$, $U\nsubseteq V$. For any $x\in X$, if there is an element $U$ of $\mathcal{U}$ such that $d(x,X\setminus U)>2\delta$, then $d(U,X\setminus U)>\delta$ and $U\cap V=\emptyset$ for any $V\in\mathcal{U}\setminus\{U\}$, in particular $x$ lies exclusively in $U$.
\end{lem}

\begin{proof}
	Observe that for any $x\in X$ and $U\subsetneq X$ open, then $d(x,X\setminus U)$ is non-zero if and only if $x\in U$. Therefore, as $\delta>0$, if $x\in X$ and $U\in\mathcal{U}$ with $d(x,X\setminus U)>2\delta$, then $x\in U$. We see that $d(U,X\setminus U)>\delta$ by applying the triangle inequality while bearing in mind that the mesh constraint on $\mathcal{U}$ means that, for any $y\in U$, $d(x,y)\leq \delta$.
	To see that $U$ intersects no other element of $\mathcal{U}$, observe that, if, for some $V\in\mathcal{U}$, there exists $y\in U\cap V$, then, for any $v\in V$, $d(y,v)\leq \delta$ by the mesh constraint on $\mathcal{U}$, but $d(U,X\setminus U)>\delta$ so $v\in U$, meaning $V\subset U$. Finally, by the non-redundancy restriction on $\mathcal{U}$ described in Assumption~\ref{notenonredundantcovers}, this must mean $V=U$.
\end{proof}

The message that the reader should take away from this lemma is that if ${d(x,X\setminus U)}$ is much larger than the scale of the cover $\mathcal{U}$, then $x$ lies in precisely one subset $U$ and this $U$ is isolated from the rest of $X$ at this scale.

	\begin{lem}\label{lemsumofweightsiscontrolledinnonisolatedcase}
		Suppose $X$ is a metric space and $\mathcal{U}$ is an open cover of $X$ with mesh at most $\delta<\diam(X)$, Lebesgue number at least $\xi$, and multiplicity at most $m+1$, such that, for any distinct $U,V\in\mathcal{U}$, $V\nsubseteq U$. If $x\in X$ is such that $d(x,X\setminus U)\leq 2\delta$ for every $U\in \mathcal{U}$, then 
		\[
		\xi\leq \left| \sum_{U\in\mathcal{U}} d(x,X\setminus U)\right| =  \sum_{U\in\mathcal{U}} d(x,X\setminus U) \leq 2(m+1)\delta.
		\]
	\end{lem}

	\begin{proof}
		By the Lebesgue number property, $B(x,\xi)\subseteq V$ for some $V\in \mathcal{U}$, meaning $d(x,X\setminus V)\geq \xi$. All other terms in the sum are non-negative, so the left inequality indeed holds.
		
		The middle equality is just the observation that each term in this sum is non-negative.
		
		The right inequality is the combination of the observations; that $d(x,X\setminus U)>0$ if and only if $x \in U$, that $x$ can be in at most $m+1$ members of the cover $\mathcal{U}$ because $\mathcal{U}$ has multiplicity at most $m+1$, and that $d(x,X\setminus U)$ is at most $2\delta$ for every $U\in\mathcal{U}$ by assumption.
	\end{proof}
	\begin{lem}\label{lemsumsofweightsislipschitz}
		Suppose $X$ is a metric space and $\mathcal{U}$ is an open cover of $X$ with multiplicity at most $m+1$. Then, for any $x,y\in X$, and $U,V\in \mathcal{U}$ which are not the entirety of $X$,
		\[
		\left|\sum_{U\in\mathcal{U}} d(x,X\setminus U) - \sum_{V\in\mathcal{U}} d(y,X\setminus V) \right| \leq 2(m+1) d(x,y).
		\]
	\end{lem}
	\begin{proof}
		This follows from combining the triangle inequality, the observation that there can be at most $m+1$ non-zero contributions from $x$ and at most $m+1$ non-zero contributions from $y$, and Lemma~\ref{lemweightfunctionsarelipschitz}.
	\end{proof}
	
	\begin{lem}\label{final 4}
		For any real numbers $a,b,c,d$,
		\[
		|ab-cd|\leq |a||b-d| + |a-c||d|.
		\]
	\end{lem}
	\begin{proof}
		Observe $|ab-cd|=|ab-ad+ad-cd|$ and use the triangle inequality.
	\end{proof}

	\begin{lem}\label{edgelengths}
		Let $i\geq0$, $f_{i}$, $\mathcal{U}_{i+1}$, and $\{p_{U}\mid U\in\mathcal{U}_{i+1}\}$ be as in the construction, given in Subsection~\ref{subsectionconstructingmaps}. Further assume that $f_{i}$ satisfies
		\begin{equation}\label{epsilondeltainductionassumption}
		d(x,y)\leq \delta_{i+1} \implies d(f_{i}(x),f_{i}(y))\leq \epsilon_{i+1}/2,
		\end{equation}
		for $x,y\in X$.
		If $U,V\in\mathcal{U}_{i+1}$, such that $U\cap V\neq \emptyset$, then 
		\[
		|p_{U}-p_{V}| = d(p_{U},p_{V})\leq 2\epsilon_{i+1}.
		\]
	\end{lem}
	\begin{proof}
		Assume $x\in U\cap V$. By construction of $f_{i+1}$,~\eqref{fi+1}, we have that there exists $x_{U} \in U$ and $x_{V}\in V$ such that $d(f_{i}(x_{U}),p_{U})=\epsilon_{i+1}/2$ and $d(f_{i}(x_{V}),p_{V})=\epsilon_{i+1}/2$. Also, by the assumption~\eqref{epsilondeltainductionassumption}, and noting that $d(x,x_{U}),d(x,x_{V})\leq \delta_{i+1}$, we get that $d(f_{i}(x),f_{i}(x_{U})),d(f_{i}(x),f_{i}(x_{V}))\leq\epsilon_{i+1}/2$. Combining this all together and using the triangle inequality, we get $|p_{U}-p_{V}| = d(p_{U},p_{V})\leq2\epsilon_{i+1}$, the desired result.
	\end{proof}

	\begin{proof}[Proof of Proposition~\ref{lemlocallylipschitz}]
		
	We proceed by induction.
	
	When $i=0$, $f_{0}(x) = f_{0}(y)$ for any $x,y\in X$ so $d(f_{0}(x),f_{0}(y)) = 0\leq \frac{L}{2}\frac{\epsilon_{0}}{\delta_{0}}d(x,y)$ trivially holds. Thus, the base case is verified.
	
	When $i=j$, assume that if $x,y\in X$ are such that $d(x,y)<\sigma\delta_{j}$, then 
	\[
	{d(f_{j}(x),f_{j}(y))\leq \frac{L}{2}\frac{\epsilon_{j}}{\delta_{j}}d(x,y).}
	\]
	In particular, plugging in the definition of $\delta_{j+1}$, given in~\eqref{deltai+1}, we get
	\begin{equation}\label{loclip induction assumption}
	d(x,y)\leq \delta_{j+1} \implies d(f_{j}(x),f_{j}(y))\leq \epsilon_{j+1}/2.
	\end{equation}
	In fact, this implication is all we will actually use from the induction hypothesis.
	
	When $i=j+1$, if we pick $x,y\in X$ with $d(x,y)\leq \sigma\delta_{j+1}$ and write out explicitly $d(f_{j+1}(x),f_{j+1}(y))$ from the definition of $f_{j+1}$, given in~\eqref{fi+1}, we get
	
		\begin{equation}\label{equationexplicitdistanceexpression}
		\left|\frac{\sum_{U\in\mathcal{U}_{j+1}} d(x,X\setminus U)p_{U}}{\sum_{U\in\mathcal{U}_{j+1}} d(x,X\setminus U)} - \frac{\sum_{U\in\mathcal{U}_{j+1}} d(y,X\setminus U)p_{U}}{\sum_{U\in\mathcal{U}_{j+1}} d(y,X\setminus U)} \right|
		\end{equation}
		
		If $d(x,X\setminus U)>2\delta_{j+1}$ or $d(y,X\setminus U)>2\delta_{j+1}$ for any $U\in\mathcal{U}_{j+1}$, then, by Lemma~\ref{lemisolatedopensetsareokay}; $x$ or $y$ is exclusively in $U$, $U\cap V = \emptyset$ for all $V\in\mathcal{U}\setminus\{U\}$, and $d(U,X\setminus U)>\delta$. We can use Lemma~\ref{lemisolatedopensetsareokay} here because of the non-redundancy assumption on $\mathcal{U}$ justified in Assumption~\ref{notenonredundantcovers}. As $x$ and $y$ are close, $d(x,y)\leq \sigma\delta_{j+1}\leq \delta_{j+1}$, we have that both $x$ and $y$ lie in $U$. Therefore both $x$ and $y$ lie exclusively in $U$, meaning $d(x,X\setminus U)$ and $d(y,X\setminus U)$ are the only non-zero terms within the sums in Expression~\eqref{equationexplicitdistanceexpression}. This means both $f_{j+1}(x)$ and $f_{j+1}(y)$ evaluate to $p_{U}$ and, therefore, this distance is equal to zero, which trivially satisfies the desired bound.
		
		Therefore, we may assume that both $d(x,X\setminus U)\leq 2\delta_{j+1}$ and $d(y,X\setminus U)\leq 2\delta_{j+1}$, for every $U\in\mathcal{U}_{j+1}$, allowing us to use Lemma~\ref{lemsumofweightsiscontrolledinnonisolatedcase} in the following manipulations of Expression~\eqref{equationexplicitdistanceexpression}. 
		
		Note that $\delta_{j+1}<1=\diam(X)$, which can be seen by combining Lemma~\ref{epsilon decreasing}, the definition of $\delta_{j+1}$ from~\eqref{deltai+1}, and $\epsilon_{0} =1$. Therefore, no element of $\mathcal{U}_{j+1}$ is the entirety of $X$, so $d(x,X\setminus U)$ is well defined, and we can apply lemmas~\ref{lemweightfunctionsarelipschitz} and~\ref{lemsumsofweightsislipschitz} in the following.

		Since $d(x,y) < \sigma\delta_{j+1}$, the subset $\{x,y\}$ of $X$ is contained in the ball $B(x,\sigma\delta_{j+1})\subseteq B(x,\mathcal{L}(\mathcal{U}_{j+1}))$, which is contained in $U_{0}\in\mathcal{U}_{j+1}$ for some $U_{0}$ by the Lebesgue number property of $\mathcal{U}_{j+1}$, and so $\{x,y\}\subset U_{0}$ too.  Let $p_{0}$ be the vertex corresponding to $U_{0}$ as in the construction, given in~\eqref{pU}.
		\[
		\eqref{equationexplicitdistanceexpression}=\left|\frac{\sum_{U\in\mathcal{U}_{j+1}} d(x,X\setminus U)p_{U}}{\sum_{U\in\mathcal{U}_{j+1}} d(x,X\setminus U)} -p_{0} +p_{0} - \frac{\sum_{U\in\mathcal{U}_{j+1}} d(y,X\setminus U)p_{U}}{\sum_{U\in\mathcal{U}_{j+1}} d(y,X\setminus U)} \right|
		\]
		Noting that 
		\[
		\frac{\sum_{U\in\mathcal{U}_{j+1}} d(x,X\setminus U)p_{0}}{\sum_{U\in\mathcal{U}_{j+1}} d(x,X\setminus U)}=p_{0},
		\]
		trivially, we get
		\[
		=\left|\frac{\sum_{U\in\mathcal{U}_{j+1}} d(x,X\setminus U)(p_{U}-p_{0})}{\sum_{U\in\mathcal{U}_{j+1}} d(x,X\setminus U)} - \frac{\sum_{U\in\mathcal{U}_{j+1}} d(y,X\setminus U)(p_{U}-p_{0})}{\sum_{U\in\mathcal{U}_{j+1}} d(y,X\setminus U)} \right|.
		\]
		Pulling out the denominators and using the lower bound from Lemma~\ref{lemsumofweightsiscontrolledinnonisolatedcase}, we observe
		\begin{align*}
		\leq\frac{1}{\sigma^{2}{\delta_{j+1}}^{2}}&\left|\sum_{U\in\mathcal{U}_{j+1}} d(y,X\setminus U)\sum_{U\in\mathcal{U}_{j+1}} d(x,X\setminus U)(p_{U}-p_{0})\right.\\
		&\left. -\sum_{U\in\mathcal{U}_{j+1}} d(x,X\setminus U)\sum_{U\in\mathcal{U}_{j+1}} d(y,X\setminus U)(p_{U}-p_{0})\right|.\\
		\end{align*}
		Using Lemma~\ref{final 4}, we see
		\begin{align*}
		\leq \frac{1}{\sigma^{2}{\delta_{j+1}}^{2}}\left(\left|\sum_{U\in\mathcal{U}_{j+1}} d(y,X\setminus U)\right|  \left|\sum_{U\in\mathcal{U}_{j+1}} (d(x,X\setminus U)-d(y,X\setminus U))(p_{U}-p_{0})\right|\right. &\\ \left.+\left|\sum_{U\in\mathcal{U}_{j+1}} (d(y,X\setminus U)-d(x,X\setminus U))\right|\left|\sum_{U\in\mathcal{U}_{j+1}} d(y,X\setminus U)(p_{U}-p_{0})\right|\right).&\\
		\end{align*}
		Now, using lemmas~\ref{lemsumofweightsiscontrolledinnonisolatedcase} and~\ref{lemsumsofweightsislipschitz}, we obtain
		\begin{align*}
		\leq \frac{1}{\sigma^{2}{\delta_{j+1}}^{2}}\Bigg(2(n+1)\delta_{j+1} & \left|\sum_{U\in\mathcal{U}_{j+1}} (d(x,X\setminus U)-d(y,X\setminus U))(p_{U}-p_{0})\right| \\ &\left.+2(n+1)d(x,y)\left|\sum_{U\in\mathcal{U}_{j+1}} d(y,X\setminus U)(p_{U}-p_{0})\right|\right).\\
		\end{align*}
		Using the triangle inequality
		\begin{align*}
		\leq \frac{1}{\sigma^{2}{\delta_{j+1}}^{2}}\Bigg(2(n+1)\delta_{j+1} & \sum_{U\in\mathcal{U}_{j+1}} \left|(d(x,X\setminus U)-d(y,X\setminus U))\right|\left|(p_{U}-p_{0})\right| \\ &\left.+2(n+1)d(x,y)\sum_{U\in\mathcal{U}_{j+1}}\left| d(y,X\setminus U)\right|\left|(p_{U}-p_{0})\right|\right).\\
		\end{align*}

		Note, for any $z\in X$, $d(z,X\setminus U)$ is non-zero if and only if $z\in U$, so the first sum has a non-zero contribution from $(p_{U}-p_{0})$ only if $x\in U$ or $y\in U$. We assumed $x,y\in U_{0}$ so, in fact, we have a non-zero contribution from $(p_{U}-p_{0})$ only if $x$ or $y$ is in $U\cap U_{0}$. Without loss of generality, assume $x\in U$. By Lemma~\ref{edgelengths} applied to $f_{j}$, we know $|p_{U}-p_{0}| = d(p_{U},p_{0})\leq2\epsilon_{j+1}$. We can apply Lemma~\ref{edgelengths} here because $f_{j}$ satisfies~\eqref{loclip induction assumption}, which was obtained from the induction assumption. Similarly, in the second sum, $|p_{U}-p_{0}|\leq2\epsilon_{j+1}$ whenever there is a non-zero contribution from $d(y,X\setminus U)$. So, we can continue to bound the above expression by
		\begin{align*}
		\leq \frac{1}{\sigma^{2}{\delta_{j+1}}^{2}}\Bigg(2(n+1)\delta_{j+1} & \sum_{U\in\mathcal{U}_{j+1}} \left|(d(x,X\setminus U)-d(y,X\setminus U))\right|2\epsilon_{j+1} \\ &\left.+2(n+1)d(x,y)\sum_{U\in\mathcal{U}_{j+1}}\left| d(y,X\setminus U)\right|2\epsilon_{j+1}\right).\\
		\end{align*}
		Finally, by lemmas~\ref{lemweightfunctionsarelipschitz} and~\ref{lemsumofweightsiscontrolledinnonisolatedcase},
		\[
		\leq\frac{1}{\sigma^{2}{\delta_{j+1}}^{2}}\bigg(2(n+1)\delta_{j+1}2(n+1)d(x,y)2\epsilon_{j+1} + 2(n+1)d(x,y)2(n+1)\delta_{j+1}2\epsilon_{j+1}\bigg)
		\]
		\[
		=\frac{16(n+1)^{2}\epsilon_{j+1}}{\sigma^{2}\delta_{j+1}}d(x,y) \leq \frac{L}{2}\frac{\epsilon_{j+1}}{\delta_{j+1}} d(x,y),
		\]
		which completes the induction.
	\end{proof}
	
\subsection{Separation}
In this subsection we show that the functions $f_{i}$ have been chosen so that points distinguished by the cover $\mathcal{U}_{i}$ remain uniformly distinguished after applying $f_{i}$. More precisely,
	
\begin{lem}\label{final separation}
	For any $i\geq 0$ and $x,y\in X$,
	\[
	\delta_{i+1}< d(x,y) \implies \frac{\epsilon_{i+1}}{\sqrt{2(n+1)}}\leq d(f_{i+1}(x),f_{i+1}(y)).
	\]
\end{lem}

First, we need a lemma.

\begin{lem}\label{final lagrange multipliers}
	The function $f\colon\mathbb{R}^{m}\rightarrow \mathbb{R}$ defined by 
	\[
	f(x_{1},x_{2},\dots,x_{m}) = \sum_{k=1}^{m}x_{k}^{2},
	\]
	for $(x_{1},\dots,x_{m})\in\mathbb{R}^{m}$, restricted to $\{(x_{1},\dots,x_{m})\mid \sum_{k=1}^{m} x_{k} = 1\text{ and }x_{j}\geq 0\text{ for all }j\}$ is minimised by $1/m$ at the point $x_{j}=1/m$ for all $j$.
\end{lem}

\begin{proof}
	This is standard. For example, one could use Lagrange multipliers.
\end{proof}

\begin{proof}[Proof of Lemma~\ref{final separation}]
	
	Recall that evaluating $f_{i+1}$ at a point $x$ is given by the sum
	\[
	\frac{\sum_{k=1}^{\left|\mathcal{U}_{i+1}\right|} d(x,X\setminus U_{k})p_{k}}{\sum_{k=1}^{\left|\mathcal{U}_{i+1}\right|} d(x,X\setminus U_{k})},
	\]
	as defined in the construction in~\eqref{fi+1}. It is key to note that the $k$-th term of the sum is non-zero if and only if $x\in U_{k}$. Now suppose $x,y\in X$ such that $d(x,y)> \delta_{i+1}$. As we chose $\mathcal{U}_{i+1}$, via Lemma~\ref{final sizecontrolledcovers}, to have mesh at most $\delta_{i+1}$, $x$ and $y$ can't both lie in any single element of $\mathcal{U}_{i+1}$.
	 Therefore, the non-zero $p_{k}$ components associated to $x$ are disjoint from the non-zero $p_{k}$ components associated to $y$. 
	Further, $f_{i+1}(x)$ will have the form $\sum_{k=1}^{m_{i}} (z_{k}e_{k}) +\sum_{k=1}^{\left|\mathcal{U}_{i+1}\right|}(\lambda_{k}( \epsilon_{i+1}/2) e_{m_{i}+k})$ where $z_{k}\in\mathbb{R}$, and $\sum_{k=1}^{\left|\mathcal{U}_{i+1}\right|}\lambda_{k} = 1$ with $\lambda_{k}\geq 0$ for all $k$ and $\lambda_{k}>0$ for at most $n+1$ distinct $k$. Similarly,
	$f_{i+1}(y)$ will have the form $\sum_{k=1}^{m_{i}} (w_{k}e_{k}) +\sum_{k=1}^{\left|\mathcal{U}_{i+1}\right|}(\mu_{k}( \epsilon_{i+1}/2) e_{m_{i}+k})$ where $w_{k}\in\mathbb{R}$, and $\sum_{k=1}^{\left|\mathcal{U}_{i+1}\right|}\mu_{k} = 1$ with $\mu_{k}\geq 0$ for all $k$ and $\mu_{k}>0$ for at most $n+1$ distinct $k$ each of which is distinct from the set of $k$ such that $\lambda_{k}>0$. This gives us a nice form to the distance between $f_{i+1}(x)$ and $f_{i+1}(y)$, namely
	\begin{align*}
	d(f_{i+1}(x),f_{i+1}(y)) &= \sqrt{\sum_{k=1}^{m_{i}} (z_{k}-w_{k})^{2} + \sum_{k=1}^{\left|\mathcal{U}_{i+1}\right|}\left(\frac{\epsilon_{i+1}\lambda_{k}}{2}\right)^{2}+ \sum_{k=1}^{\left|\mathcal{U}_{i+1}\right|}\left(\frac{\epsilon_{i+1}\mu_{k}}{2}\right)^{2}},\\
	&\geq \frac{\epsilon_{i+1}}{2} \sqrt{\sum_{k=1}^{\left|\mathcal{U}_{i+1}\right|}\lambda_{k}^{2}+ \sum_{k=1}^{\left|\mathcal{U}_{i+1}\right|}\mu_{k}^{2}}.\\
	\end{align*}
	Now by Lemma~\ref{final lagrange multipliers}, bearing in mind the conditions on the $\lambda_{k}$ and $\mu_{k}$ sums, we get
	\[
	d(f_{i+1}(x),f_{i+1}(y)) \geq \frac{\epsilon_{i+1}}{2}\sqrt{\frac{2}{(n+1)}} =\frac{\epsilon_{i+1}}{\sqrt{2(n+1)}}.\qedhere
	\]
\end{proof}
	
	\subsection{$q$-measure}
	In this subsection we show the functions $f_{i}$ have been chosen so that we can easily bound the $q$-measure of their images.
	
In the following we will refer to `simplicies' containing the image of $f_{i+1}$. What we mean by this is that if you consider $x\in X$, there exist $U_{j_{1}},\dots U_{j_{m}} \in\mathcal{U}_{i+1}$, such that $x\in U_{j_{k}}$ for all $k$. Then $f_{i+1}(x)$ sits inside 
	\[
	f_{i+1}\left(\bigcap_{k=1}^{m}U_{j_{k}}\right)\subseteq [p_{j_{1}},\dots,p_{j_{m}}] = \left\{\sum_{k=1}^{m}\lambda_{k}p_{j_{k}}\right\},
	\]
	where $\lambda_{k}\geq 0$ for all $k$ and $\sum_{k=1}^{m} \lambda_{k} = 1$.
	We say that $[p_{j_{1}},\dots,p_{j_{m}}]$ is a \redit{simplex} containing $f_{i+1}(x)$. Due to our construction of $f_{i}$, the image of $f_{i}$ is contained in an $n$-dimensional simplicial complex with simplicies of diameter approximately $\epsilon_{i+1}$. Subspaces in $\mathbb{R}^{n}$ can be covered efficiently at all scales.
	Hence, our approach is to cover each component simplex efficiently using our knowledge of covering $n$-dimensional euclidean space, then union over all simplicies to cover the complex. The upper bound on the number of elements in the cover $\mathcal{U}_{i+1}$ from Lemma~\ref{final sizecontrolledcovers} will give control on the number of simplicies in this simplicial complex.
	
	For each $i\geq 0$ define
	\begin{equation}\label{etai+1}
	\eta_{i+1} \coloneqq 8\epsilon_{i+2}.
	\end{equation}
	This will be the scale within $\ell^{2}$ at which we cover $f_{i+1}(X)$, the $(i+1)$-th approximation of $X$.
	
	\begin{lem}\label{qmeasure}
		For each $i\geq 0$, there exists a cover, $\mathcal{V}_{i+1}$, of $f_{i+1}(X)$ with mesh at most $4\eta_{i+1}$, Lebesgue number at least $\eta_{i+1}$ as a cover of $f_{i+1}(X)\subseteq \ell^{2}$, and
		\[
		\sum_{V\in\mathcal{V}_{i+1}} \diam(V)^{q}\leq 4^{q}.
		\] 
	\end{lem}

The `Lebesgue number' component on this lemma is present so that $\mathcal{V}_{i+1}$ also covers a small $\ell^{2}$-neighbourhood of the image. Hence, $\mathcal{V}_{i+1}$ will also cover the image of functions similar to $f_{i+1}$. In particular, this will mean that the image of the pointwise limit of $(f_{j})_{j\in\mathbb{N}}$ will be covered by $\mathcal{V}_{i+1}$, and this will give us a useful family of covers for computing the Hausdorff $q$-measure of this limit image.
	
	\begin{proof}[Proof of Lemma~\ref{qmeasure}]
Following our strategy detailed above, we begin by explaining how to cover an individual simplex containing some of the image of $f_{i+1}$.
		 
		For each simplex, $\Delta=[v_{0},\dots,v_{m}]$ with $v_{j}\in \{p_{U}\mid U\in\mathcal{U}_{i+1}\}$ and $m\leq n$, by picking a base vertex of $\Delta$, say $v_{0}$, and letting $S$ be the span of the vectors $(v_{j}-v_{0})$ in $\ell^{2}$, we observe that $\Delta$ sits inside a translated copy of $\mathbb{R}^{r}$, say $(S+v_{0})$, for some $r\leq m \leq n$. Lemma~\ref{edgelengths} tells us that the edge lengths of $\Delta$ are at most $2\epsilon_{i+1}$, hence we see that $\Delta$ is contained in
		an $r$-cube in $(S+v_{0})$ of edge length $4\epsilon_{i+1}$ centred at $v_{0}$. To be precise, identify $S$ with $\mathbb{R}^{r}$, let $R=[-2\epsilon_{i+1},2\epsilon_{i+1}]^{r}\subset S$, and then translate $R$ so that it is centred on $v_{0}$; $(R+v_{0})$ is the aforementioned $r$-cube. Consider the cover of $R$ by subdividing $R$ into $r$-cubes of edge length $\eta_{i+1}/\sqrt{n}$. Note each of these covering $r$-cubes has diameter $\sqrt{r}\eta_{i+1}/\sqrt{n}\leq \eta_{i+1}$. This requires at most $4\sqrt{n}\epsilon_{i+1}/\eta_{i+1} + 1$ subdivisions along each edge of $R$. One may observe that ${4\sqrt{n}\epsilon_{i+1}/\eta_{i+1}\geq 1}$ by combining Lemma~\ref{epsilon decreasing} with the definition of $\eta_{i+1}$ as $8\epsilon_{i+2}$,~\eqref{etai+1}.
		Hence, we can simplify this to at most $8\sqrt{n}\epsilon_{i+1}/\eta_{i+1}$ subdivisions. Thus, to cover the whole $r$-cube $R$ we need at most $(8\sqrt{n}\epsilon_{i+1}/\eta_{i+1})^{r}$ $r$-cubes of edge length $\eta_{i+1}/\sqrt{n}$. As $r\leq n$ and $4\sqrt{n}\epsilon_{i+1}/\eta_{i+1}\geq 1$, we can see that $(8\sqrt{n}\epsilon_{i+1}/\eta_{i+1})^{r}\leq(8\sqrt{n}\epsilon_{i+1}/\eta_{i+1})^{n}$.  Now, if we do this for each simplex in the complex $f_{i+1}(X)$ and take the collection of all these covering hypercubes of diameter at most $\eta_{i+1}$, we get a cover, say $\mathcal{V}_{i+1}$, for $f_{i+1}(X)$.
		
		We picked $\mathcal{U}_{i+1}$ from Lemma~\ref{final sizecontrolledcovers}, so we know that 
		\[
		|\mathcal{U}_{i+1}|\leq N^{\log_{2}(2/\sigma\delta_{i+1})}.
		\]
		We also know that every simplex is a choice of at most $(n+1)$ distinct elements of $\mathcal{U}_{i+1}$, which is at most a choice of $(n+1)$ elements of $\mathcal{U}_{i+1}$ with repetition. Hence, if we define $\Delta_{i+1}$ to be the set of all simplicies, then
		\begin{align}
		|\Delta_{i+1}|&\leq \sum_{m=0}^{n}\left|\mathcal{U}_{i+1}\right|^{m+1},\nonumber\\
		&\leq \left|\mathcal{U}_{i+1}\right|^{n+2},\nonumber\\
		&\leq N^{(n+2)\log_{2}(2/\sigma\delta_{i+1})}.\label{Deltai+1}
		\end{align}
		We can, therefore, estimate the $q$-measure of the image by
		\[
		\sum_{V\in\mathcal{V}_{i+1}} \diam (V)^{q} \leq |\Delta_{i+1}| (8\sqrt{n}\epsilon_{i+1}/\eta_{i+1})^{n} {\eta_{i+1}}^{q},
		\]
		with the cover described above.
		
		Noting that $\epsilon_{i+1}\leq 1$, by Lemma~\ref{epsilon decreasing} and $\epsilon_{0}=1$, we can again bound by
		\begin{equation}\label{equationpartialmeasurecalculation}
		\sum_{V\in\mathcal{V}_{i+1}} \diam (V)^{q} \leq |\Delta_{i+1}| (8\sqrt{n})^{n} \eta_{i+1}^{q-n}.
		\end{equation}
		Plugging the definition of $\epsilon_{i+2}$,~\eqref{epsiloni+1}, into the definition of $\eta_{i+1}$,~\eqref{etai+1}, and combining with equations~\eqref{Deltai+1} and~\eqref{equationpartialmeasurecalculation} we obtain 
		\[
		\sum_{V\in\mathcal{V}_{i+1}} \diam (V)^{q} \leq 1.
		\]
		Now we have a cover of the image of $f_{i+1}$, but we are lacking control on the Lebesgue number of $\mathcal{V}_{i+1}$ to ensure that $\mathcal{V}_{i+1}$ also covers a small $\ell^{2}$-neighbourhood of the image. For each $V\in\mathcal{V}_{i+1}$, pick a point $y_{V}\in V$, then consider $V\subset\overbar{B(y_{V},\eta_{i+1})}$ and $\diam\left(\overbar{B(y_{V},\eta_{i+1})}\right) \leq  2\diam (V)$ so we can replace $V$ by $\overbar{B(y_{V},\eta_{i+1})}$ and still cover $f_{i+1}(X)$ without significantly affecting the value of the above sum. We can also double the radius of $\overbar{B(y_{V},\eta_{i+1})}$ without changing the value of the above sum by much, so if we consider a new cover $\mathcal{V}_{i+1}^{\prime} = \{\overbar{B(y_{V},2\eta_{i+1})}\mid V\in\mathcal{V}\}$. Then
		\[
		\sum_{V^{\prime}\in\mathcal{V}_{i+1}^{\prime}} \diam (V^{\prime})^{q} \leq \sum_{V\in\mathcal{V}_{i+1}} (4\diam (V))^{q} \leq 4^{q}.
		\]
		Now, we proceed to show that we have good control of the Lebesgue number of this modified cover. Indeed, for any $y\in f_{i+1}(X)$, $\mathcal{V}$ is a cover of $f_{i+1}(X)$, so there exists some $V\in\mathcal{V}$ such that $y\in V\subset \overbar{B(y_{V},\eta_{i+1})}$ and thus $B(y,\eta_{i+1})\subset\overbar{B(y_{V},2\eta_{i+1})}\in\mathcal{V}_{i+1}^{\prime}$, so $\mathcal{L}(\mathcal{V}_{i+1}^{\prime})\geq \eta_{i+1}$. Therefore, as the mesh of $\mathcal{V}_{i+1}^{\prime}$ is at most $4\eta_{i+1}$, we see that $\mathcal{V}_{i+1}^{\prime}$ satisfies the requirements for our lemma.
	\end{proof}

	\subsection{Cauchy}
	In this section, we justify that we have good control on how close each map in the sequence of maps $(f_{j})_{j\in\mathbb{N}}$ is to the previous map.
	\begin{lem} \label{Cauchy}
		$d(f_{i+1},f_{i})\leq \epsilon_{i+1}$
	\end{lem} 
The following proof is essentially from~\cite[page 59]{hurewicz2015dimension}.
\begin{proof}
	For $x\in X$, let $U_{j_{1}},\dots,U_{j_{m}}$ be the collection of elements of $\mathcal{U}_{i+1}$ that contain $x$. Recall from Definition~\eqref{pU} that, for each $1\leq k\leq m$, $U_{j_{k}}$ is associated to a point, $x_{j_{k}}$, contained within. As both $x$ and $x_{j_{k}}$ are contained in $U_{j_{k}}$, and $\diam(U_{j_{k}})\leq \delta_{i+1}$, we see that $d(x,x_{j_{k}})\leq \delta_{i+1}$. Further, if we apply Corollary~\ref{controlledstretching}, we deduce $d(f_{i}(x),f_{i}(x_{j_{k}}))\leq \epsilon_{i+1}/2$. By our choice of $p_{j_{k}}$, we know $d(f_{i}(x_{j_{k}}),p_{j_{k}})= \epsilon_{i+1}/2$, and therefore $|p_{j_{k}}-f_{i}(x)| = d(f_{i}(x),p_{j_{k}})\leq \epsilon_{i+1}$ follows by application of the triangle inequality. Hence,
	\begin{align*}
	d(f_{i+1}(x),f_{i}(x))&= \left|\left(\frac{\sum_{k=1}^{\left|\mathcal{U}_{i+1}\right|} d(x,X\setminus U_{k})p_{k}}{\sum_{k=1}^{\left|\mathcal{U}_{i+1}\right|} d(x,X\setminus U_{k})}\right)-f_{i}(x)\right|\\
			  &= \left|\frac{\sum_{k=1}^{m}d(x,X\setminus U_{j_{k}})(p_{j_{k}}-f_{i}(x))}{\sum_{k=1}^{m}d(x,X\setminus U_{j_{k}})}\right|\\
			  &\leq \frac{\sum_{k=1}^{m}d(x,X\setminus U_{j_{k}})|p_{j_{k}}-f_{i}(x)|}{\sum_{k=1}^{m}d(x,X\setminus U_{j_{k}})}\\
			  &\leq \frac{\sum_{k=1}^{m}d(x,X\setminus U_{j_{k}})\epsilon_{i+1}}{\sum_{k=1}^{m}d(x,X\setminus U_{j_{k}})}\\
			  &=\epsilon_{i+1}.\qedhere
	\end{align*}
\end{proof}

Along with our knowledge of the sequence $(\epsilon_{j})_{j\in\mathbb{N}}$ from Lemma~\ref{epsilon decreasing}, we see that the sequence $(f_{j})_{j\in\mathbb{N}}$ is Cauchy.

\section{The bi-H\"older embedding 
}
In this section we define the map $f$ mentioned in our main theorem, Theorem~\ref{thmmain}, and verify it has the desired properties for the theorem.

	\subsection{The definition and basic properties of $f$}
In this subsection we define the map $f$ as the pointwise limit of the sequence $(f_{j})_{j\in \mathbb{N}}$, prove its existence, and show how nice properties of $f_{i}$ translate over to $f$.

\begin{lem}\label{final dist to pointwise limit}
	The pointwise limit, $f$, of $\{f_{j}\}_{j\in\mathbb{N}}$ exists, and, for $i\geq 1$, $d(f,f_{i-1})\leq 2\epsilon_{i}$.
\end{lem}

\begin{proof}
	Recall, Lemma~\ref{Cauchy} gives
	\[
	d(f_{j},f_{j+1})\leq \epsilon_{j+1},\text{ for all }j,
	\]
	and Lemma~\ref{epsilon decreasing} gives
	\[
	\epsilon_{j+1}\leq \frac{1}{L}\epsilon_{j}, \text{ for all }j,
	\]
	and $L\geq 2$. If we combine these with the triangle inequality, we may observe that $(f_{j}(x))_{j\in\mathbb{N}}$ is a Cauchy sequence in $\ell^{2}$ for every $x\in X$ and therefore, by the completeness of $\ell^{2}$, has a limit. Define $f\colon X\rightarrow \ell^{2}$ by $f(x) = \lim\limits_{j\to \infty} f_{j}(x)$ for $x\in X$.
	Now for each $x\in X$ and $i\geq 1$, $d(f(x),f_{i-1}(x)) = \lim\limits_{j\to\infty}d(f_{j}(x),f_{i-1}(x))$, but, for $j\geq i$, by repeated use of the triangle inequality and noting that $\epsilon_{j+1}\leq \epsilon_{j}/L\leq \epsilon_{j}/2$, for all $j\geq 0$, we get
	\[
	d(f_{i-1}(x),f_{j}(x))\leq \sum_{k=i}^{j} d(f_{k-1}(x),f_{k}(x)) \leq \sum_{k=i}^{j} \epsilon_{k}\leq \sum_{k=0}^{j-i}\frac{\epsilon_{i}}{2^{k}} = \epsilon_{i}\left( 2 - \frac{1}{2^{j-i}}\right).
	\]
	Therefore, taking $j\to\infty$, we see that $d(f,f_{i-1})\leq 2\epsilon_{i}$.
\end{proof}

We see in the following two lemmas that the control imposed on the approximating functions, $f_{i}$, roughly follows through to the limit.
\begin{lem}\label{final above}
	For all $i\geq 0$,
	\[
	d(x,y)\leq \delta_{i} \implies d(f(x),f(y))\leq \frac{9}{2}\epsilon_{i}.
	\]
\end{lem}

\begin{proof}

	When $i\geq 1$, using the triangle inequality combined with $d(f,f_{i-1})\leq 2\epsilon_{i}$ from Lemma~\ref{final dist to pointwise limit}, and
	\[
	d(x,y)\leq \delta_{i} \implies d(f_{i-1}(x),f_{i-1}(y))\leq \frac{1}{2}\epsilon_{i},
	\]
	from Corollary~\ref{controlledstretching}, we get the desired
	\[
	d(x,y)\leq \delta_{i} \implies d(f(x),f(y))\leq \frac{9}{2}\epsilon_{i}.
	\]
	
	For $i=0$ we similarly use the triangle inequality and $d(f,f_{0})\leq 2\epsilon_{1}\leq \epsilon_{0}$, but instead of using Corollary~\ref{controlledstretching} we use the definition of $f_{0}$ to get $d(f_{0}(x),f_{0}(y)) = 0$ for all $x,y\in X$.
\end{proof}

\begin{lem}\label{final below}
	For all $i\geq 0$,
	\[
	 \delta_{i+1}<d(x,y) \implies  \frac{1}{2\sqrt{2(n+1)}}\epsilon_{i+1}\leq d(f(x),f(y)).
	\]
\end{lem}

\begin{proof}
	By Lemma~\ref{final separation}, we have 
	\[
	 \delta_{i+1}<d(x,y) \implies  \frac{1}{\sqrt{2(n+1)}}\epsilon_{i+1}\leq d(f_{i+1}(x),f_{i+1}(y)).
	\]
	From Lemma~\ref{final dist to pointwise limit}, we have $d(f,f_{i+1})\leq 2\epsilon_{i+2}$. From Lemma~\ref{epsilon decreasing} we have $\epsilon_{i+2}\leq  \epsilon_{i+1}/L$, and note that $L\geq (8\sqrt{2(n+1)})$, so
	\[
	\epsilon_{i+2}\leq \frac{1}{L} \epsilon_{i+1}\leq \frac{1}{8\sqrt{2(n+1)}}\epsilon_{i+1}.
	\]
	Combining these facts and using the triangle inequality, we get, for any $x,y\in X$ such that $\delta_{i+1}<d(x,y)$,
	\[
	 \frac{1}{2\sqrt{2(n+1)}}\epsilon_{i+1}=\frac{1}{\sqrt{2(n+1)}}\epsilon_{i+1}-2(2\frac{1}{8\sqrt{2(n+1)}}\epsilon_{i+1})\leq d(f(x),f(y)).\qedhere
	\]
\end{proof}

\subsection{The Hausdorff dimension of $f(X)$ is at most $q$}

\begin{prop}\label{propqmeasureatmost4q}
	The image, $f(X)$, of $f$ has Hausdorff dimension at most $q$, in particular $f(X)$ has $q$-Hausdorff measure at most $4^{q}$. 
\end{prop}

\begin{proof}
	By Lemma~\ref{qmeasure}, for every $i\geq 1$, we have a cover, $\mathcal{V}_{i}$, of the image $f_{i}(X)$, with mesh at most $4\eta_{i}$ and  Lebesgue number, as a cover of $f_{i}(X)\subseteq \ell^{2}$, at least $\eta_{i}$ such that
	\[
	\sum_{V\in\mathcal{V}_{i}} \diam(V)^{q}\leq \sum_{V\in\mathcal{V}_{i}} (4\eta_{i})^{q}\leq 4^{q}.
	\]
Plugging the definition of $\eta_{i+1}$, given in~\eqref{etai+1}, into the result of Lemma~\ref{final dist to pointwise limit} we find $d(f,f_{i})\leq \eta_{i}/4$.
Hence, for any $x\in X$, $f(x)$ lies in the ball $B(f_{i}(x),\eta_{i}/2)$ which is contained in an element of $\mathcal{V}_{i}$ by the Lebesgue number component of Lemma~\ref{qmeasure}. Therefore, $\mathcal{V}_{i}$ is also a cover for $f(X)$.
	
	Noting that $\eta_{i}\to 0$, because $\epsilon_{i}\to 0$ by Lemma~\ref{epsilon decreasing}, we see that the Hausdorff $q$-measure of $f(X)$ is at most $4^{q}$, which is finite so the Hausdorff dimension of $f(X)$ is at most $q$.
\end{proof}

\subsection{The map $f$ is bi-H\"older onto its image}

\begin{prop}\label{biHolder}
	The map $f$ is bi-H\"older onto its image, in particular there exists $Q=Q(n,q,N)$ and $\lambda = \lambda(n,q,N,\sigma)$ such that, for any $x,y\in X$,
	\begin{equation}\label{equationbiholderprop}
	\frac{1}{\lambda}d(x,y)^{2Q}\leq d(f(x),f(y))\leq \lambda d(x,y)^{\frac{1}{4Q}}.
	\end{equation}
	
\end{prop}

\begin{lem}\label{lessL}
	For all $i\geq 0$,
	\[
	\epsilon_{i}\leq \frac{1}{L^{i}}.
	\]
\end{lem}

\begin{proof}
Recall $\epsilon_{0} =1$ so applying Lemma~\ref{epsilon decreasing} inductively, we get
	\[
	\epsilon_{i}\leq \left(\frac{1}{L}\right)^{i}\epsilon_{0}= \frac{1}{L^{i}}.\qedhere
	\]
\end{proof}

As we did in Lemma~\ref{epsilon decreasing}, it is convenient to simplify our definition of $\epsilon_{i+1}$ as given in~\eqref{epsiloni+1}. Recall
\[
\epsilon_{i+1} =  \frac{1}{8}\left((8\sqrt{n})^{n}N^{(n+2)\log_{2}(2/\sigma\delta_{i})}\right)^{\frac{-1}{q-n}}.
\]
Let $C=C(n,q,N,\sigma)$, be defined as
\begin{equation}\label{defnC}
C\coloneqq 8(8\sqrt{n})^{\frac{n}{q-n}}N^{\frac{n+2}{q-n}\log_{2}(2/\sigma)}=\frac{1}{\epsilon_{1}},
\end{equation}
and $Q=Q(n,q,N)$ as
\begin{equation}\label{defnQ}
Q\coloneqq \frac{n+2}{q-n}\log_{2}(N),
\end{equation}
and note that $C,Q$ are large positive constants.
So now we can write
\begin{equation}\label{simpleepsiloni+1}
	\epsilon_{i+1} = \frac{1}{C}\delta_{i}^{Q}.
\end{equation}

\begin{lem}\label{epsilondontshrinktoofast}
	For all $i\geq0$,
	\[
	\epsilon_{i}\leq (C\epsilon_{i+1})^{\frac{1}{2Q}}.
	\]
\end{lem}
\begin{proof}
	By plugging Lemma~\ref{final epsilondeltarelation} into~\eqref{simpleepsiloni+1}, we see
	\[
	\epsilon_{i+1} = \frac{1}{C}\left(\frac{1}{L^{i}}\epsilon_{i}\right)^{Q}.
	\]
	Using Lemma~\ref{lessL}, we observe
	\[
	\frac{1}{C}\epsilon_{i}^{2Q}\leq 	 \frac{1}{C}\left(\frac{1}{L^{i}}\epsilon_{i}\right)^{Q}= \epsilon_{i+1}.
	\]
	Rearranging, we get
	\[
	\epsilon_{i}\leq (C\epsilon_{i+1})^{\frac{1}{2Q}},
	\]
	because $C,Q\geq 1$ from Note~\ref{Nbig}, as desired.
\end{proof}

\begin{proof}[Proof of Proposition~\ref{biHolder}]

As a consequence of Lemma~\ref{epsilon decreasing} and~\eqref{deltai+1} we see that $(\delta_{j})_{j\in\mathbb{N}}$ limits to zero, this sequence starts at $\delta_{0}=1$ and thus for any $r\in(0,1]$ there exists $i$ such that
\[
\delta_{i+1}<r\leq \delta_{i}.
\]
Also, as $\diam(X)=1$, for any $x,y\in X$, either $x=y$, or $d(x,y)\in(0,1]$ and, therefore, there exists $i$ such that
\[
\delta_{i+1}<d(x,y)\leq \delta_{i}.
\]
If $x=y$, then~\eqref{equationbiholderprop} trivially holds, so we assume $\delta_{i+1}<d(x,y)\leq \delta_{i}$.
This implies that 
	\[
  \frac{1}{2\sqrt{2(n+1)}}\epsilon_{i+1}\leq d(f(x),f(y))\leq \frac{9}{2}\epsilon_{i},
\]
by combining Lemma~\ref{final above} and Lemma~\ref{final below}.
We can weaken the bounds on $d(x,y)$ using Lemma~\ref{lessL} and Lemma~\ref{final epsilondeltarelation}, to get
\[
{\epsilon_{i+1}}^{2}\leq \frac{1}{L^{i+1}}\epsilon_{i+1} = \delta_{i+1} <d(x,y),
\]
and
\[
d(x,y)\leq \delta_{i}=\frac{1}{L^{i}}\epsilon_{i}\leq \epsilon_{i}.
\]
Now, utilising Lemma~\ref{epsilondontshrinktoofast}, we get
\[
d(f(x),f(y))\leq\frac{9}{2}\epsilon_{i}\leq \frac{9}{2}C^{\frac{1}{2Q}}({\epsilon_{i+1}}^{2})^{\frac{1}{4Q}}\leq\frac{9}{2}C^{\frac{1}{2Q}}(d(x,y))^{\frac{1}{4Q}},
\]
and
\[
\frac{1}{2\sqrt{2(n+1)}C}d(x,y)^{2Q}\leq\frac{1}{2\sqrt{2(n+1)}C}{\epsilon_{i}}^{2Q}\leq \frac{1}{2\sqrt{2(n+1)}}\epsilon_{i+1}\leq d(f(x),f(y)).
\]
Let $\lambda=\lambda(n,C,Q)$ be defined as
\[
\lambda \coloneqq \max\left\{\frac{9}{2}C^{\frac{1}{2Q}},2\sqrt{2(n+1)}C\right\}.
\]
Summarising, we have the desired bi-H\"older inequality for $f$
\[
	\frac{1}{\lambda}d(x,y)^{2Q}\leq d(f(x),f(y))\leq \lambda d(x,y)^{\frac{1}{4Q}}.\qedhere
\]
\end{proof}

\subsection{Checking the consistency of choices relating to $N$}

Throughout, we made some assumptions on the size of $N$. In this subsection, we summarise these assumptions and verify that they can be satisfied simultaneously.

Recall the following definitions of constants:
\begin{alignat*}{2}
L&= \frac{128(n+1)^{2}}{\sigma^{2}}, &&\text{ \eqref{defnL}}\\
\delta_{1}&= \frac{1}{L}\epsilon_{1}, &&\text{ \eqref{deltai+1}}\\ 
B_{1}&= (8\sqrt{n})^{\frac{n}{q-n}}, &&\text{ \eqref{defnB1}}\\
B_{2}&= \frac{n+2}{q-n}\log_{2}\left(\frac{2}{\sigma}\right), &&\text{ \eqref{defnB2}}\\
C&= 8(8\sqrt{n})^{\frac{n}{q-n}}N^{\frac{n+2}{q-n}\log_{2}(2/\sigma)} = \frac{1}{\epsilon_{1}}, &&\text{ \eqref{defnC}}\\
Q&= \frac{n+2}{q-n}\log_{2}(N). &&\text{ \eqref{defnQ}}
\end{alignat*}

\begin{lem}\label{assumptionsonN}
	For $N$ sufficiently large,
	\begin{gather*}
	\delta_{1} \text{ is sufficiently small to apply Lemma~\ref{final sizecontrolledcovers}},\\
	C\geq 1,\\
	Q\geq 1,\\
	B_{1}N^{B_{2}}\geq L.
	\end{gather*}
\end{lem}

\begin{proof}
	Note that $n+2\geq 1$ and $q-n>0$ so $(n+2)/(q-n)> 0$. Hence, $Q\to \infty$ as $N\to \infty$. Further, $0<\sigma<1$ so $\log_{2}(2/\sigma)>0$, and, therefore $C\to \infty$ as $N\to \infty$ as well.
	
	Note that $B_{1}$ and $B_{2}$, are both positive and independent of $N$, hence
	$B_{1}N^{B_{2}}\to \infty$ as $N\to \infty$.
	Observe that $L$ is independent of $N$ and therefore $B_{1}N^{B_{2}}\geq L$ for $N$ large enough. This independence between $L$ and $N$ also has the consequence that $\delta_{1} = 1/(LC) \to 0$ as $N\to \infty$.
\end{proof}

\subsection{Proof of Theorem~\ref{thmmain}}

We now summarise the above to conclude the proof of Theorem~\ref{thmmain}.

\begin{proof}[Proof of Theorem~\ref{thmmain}]
	Suppose $\diam(X) = 1$.
	The $f$ stated in the theorem is the pointwise limit of the sequence of functions $(f_{j})_{j\in\mathbb{N}}$ as defined in Section~\ref{constructionsec}. By Proposition~\ref{propqmeasureatmost4q}, the Hausdorff $q$-measure of $f(X)$ is at most $4^{q}$, and, by Proposition~\ref{biHolder}, the distance estimate holds with $\mu=\lambda$, $\alpha = 2Q$ and $\beta = 1/4Q$.
	
	For any bounded metric space $(X,d)$ with non-zero diameter, we can always rescale the metric and get another metric space with much the same properties but with diameter equal to $1$. Indeed, note that $(X,\frac{1}{\diam(X)}d)$ is also a metric space which is compact if $X$ is; similarly, doubling and capacity dimension $n$ with coefficient $\sigma$ are unaffected by rescaling the metric. Let
	\[
	\phi\colon (X,d)\rightarrow \left(X,\frac{1}{\diam(X)}d\right)
	\]
	be the identity map on the set $X$. As $\left(X,\frac{1}{\diam(X)}d\right)$ has diameter equal to $1$, from the above, we have a map $f\colon \left(X,\frac{1}{\diam(X)}d\right)\rightarrow \ell^{2}$ satisfying inequalities
	\[
	\frac{1}{\mu}\left(\frac{1}{\diam(X)}d(x,y)\right)^{\alpha}\leq d_{\ell^{2}}(f(x),f(y))\leq \mu \left(\frac{1}{\diam(X)}d(x,y)\right)^{\beta},
	\]
	for some $\mu,\alpha,\beta$. Pull this back to $(X,d)$ via $\phi$ to get $f\circ\phi\colon (X,d)\rightarrow \ell^{2}$ with
	\[
	\frac{1}{\mu\diam(X)^{\alpha}}d(x,y)^{\alpha}\leq d_{\ell^{2}}(f\circ\phi(x),f\circ\phi(y))\leq \frac{\mu}{\diam(X)^{\beta}} d(x,y)^{\beta}.
	\]
	As $\phi$ is the identity map on the set $X$, the image of $f$ doesn't change under composition with $\phi$, $f\circ\phi(X) = f(X)$, and therefore the image of $f\circ\phi$ also has Hausdorff $q$-measure at most $4^{q}$.
\end{proof}

\subsection{Proofs of corollaries}

\begin{proof}[Proof of Corollary~\ref{cor1}]
	From Theorem~\ref{thmmain}, for any $q>n$, we can find a bi-H\"older map, $f\colon X\rightarrow \ell^{2}$, where the image has finite Hausdorff $q$-measure and, therefore, Hausdorff dimension at most $q$. Restricting the range of $f$ to its image makes $f$ a H\"older equivalence between $X$ and $f(X)$, which gives $\holdim(X)\leq q$. Further, as $q>n$ was arbitrary, we conclude $\holdim(X)\leq n$.
\end{proof}

To relate Theorem~\ref{thmmain} to Corollary~\ref{cormain} we need to understand local self-similarity better.

\begin{lem}\label{doubling}
	If $X$ is a compact, locally self-similar metric space, then $X$ is $N$-doubling, for some $N\in\mathbb{N}$.
\end{lem}

\begin{proof}
	Let $\lambda, R,$ and $\Lambda_{0}$ be as in the definition of local self-similarity, Definition~\ref{defnlocselfsim}.
	Let $\epsilon = \min \{\Lambda_{0}/2R, \Lambda_{0}/16\lambda  \}$. Consider the cover of $X$, $\bigcup_{x\in X} B(x,\epsilon)$, by $\epsilon$-balls and use compactness to take a finite subcover, say $\bigcup_{i=1}^{N} B(x_{i},\epsilon)$. Let $x \in X$ and $r>0$. If $r\geq \Lambda_{0}/R$ then, as $\epsilon \leq \Lambda_{0}/2R\leq r/2$, $\{B(x_{i},r/2)\}_{i=1}^{N}$ is a cover of $B(x,r)$, as it covers $X$, of at most $N$ balls of half the radius. If $r<\Lambda_{0}/2R$, then $\Lambda_{0}/2r\geq R$, so $\Lambda_{0}/2r$ is sufficiently large to apply local self-similarity of $X$. As $\diam(B(x,r))\leq 2r\leq \Lambda_{0}/(\Lambda_{0}/2r)$, we can find a $\lambda$-bi-Lipschitz embedding, $f\colon(B(x,r),\frac{\Lambda_{0}}{2r}d)\rightarrow X$, by the local self-similarity of $X$. Note that as $\{B(x_{i},\epsilon)\}_{i=1}^{N}$ covers $X$, it also covers $f(B(x,r))$, so we can pull this cover of the image back through $f$ to get a cover for $B(x,r)$. Using the lower bound of the bi-Lipschitz inequality for $f$, we get that the diameter of the preimage of each element of this cover is at most $4\lambda r \epsilon/\Lambda_{0}$. If the preimage of a ball is empty, then it contributes nothing to covering $B(x,r)$ and we can ignore it. If not, pick $y_{i} \in f^{-1}(B(x_{i},\epsilon))$ and note that $f^{-1}(B(x_{i},\epsilon)) \subseteq B(y_{i},8\lambda r \epsilon/\Lambda_{0})$. Hence, $\{B(y_{i},8\lambda r \epsilon\Lambda_{0})\}_{i}$ covers $B(x,r)$. As $\epsilon \leq \Lambda_{0}/16\lambda$, the radius of each of these balls is at most $r/2$. There are at most $N$ centres, $y_{i}$, so at most $N$ balls in this collection. Hence, $B(x,r)$ is covered by at most $N$ balls of half the radius and we obtain the desired result. 
\end{proof}

\begin{proof}[Proof of Corollary~\ref{cormain}.]
	From Lemma~\ref{doubling}, $X$ locally self-similar implies $X$ is doubling, so we can use Corollary~\ref{cor1} to obtain that $X$ has H\"older dimension at most its capacity dimension. By Theorem~\ref{BLcor}, $X$ has capacity dimension equal to its topological dimension therefore, the H\"older dimension of $X$ is at most its topological dimension. However, topological dimension is always a lower bound for H\"older dimension by~\eqref{topdimlowerboundforholdim}, so we have the reverse inequality too. Combining these inequalities, we arrive at the desired equality.
\end{proof}

\section{Notation for Cantor sets}\label{sectionnotationforcantorsets}
In this section, we provide notation for a standard construction of Cantor sets in the interval $[0,1]$. This notation is an extension of that which is found in~\cite[Section 4.10]{Mattila}. 

Denote $I=I_{0,1}=[0,1]$. To construct the standard $1/3$-Cantor set, cut the middle third, $J_{1,1}=(1/3,2/3)$, from $I_{0,1}$ to get $I_{1,1} = [0,1/3]$ and $I_{1,2}= [2/3,1]$, then repeat by cutting out the middle third of both $I_{1,1}$ and $I_{1,2}$, and so on. The remaining set, after cutting out middle thirds ad infinitum in this manner, is the $1/3$-Cantor set. More precisely, inductively define a collection of intervals $I_{n,i}$, for $n\in\mathbb{N}$ and $1\leq i \leq 2^{n}$, by removing the middle open interval, $J_{n+1,i}$, of diameter $\frac{1}{3}\diam(I_{n,i})$ from $I_{n,i}$ to get two disjoint closed intervals $I_{n+1,2i-1}$ and $I_{n+1,2i}$. Taking the collection of points that, for every $n\geq 0$, lie in one of $I_{n,i}$ gives us a Cantor set, $C$,
\[
C=\bigcap_{n=0}^{\infty}\bigcup_{i=1}^{2^{n}}I_{n,i}.
\] 
This Cantor set is known as the $1/3$-Cantor set because of the constant ratio of diameters
\[
\frac{\diam(J_{n+1,i})}{\diam(I_{n,i})} = \frac{1}{3}.
\]
However, we could have chosen $J_{n+1,i}$ to be the middle open interval of diameter $\lambda\diam(I_{n,i})$ from $I_{n,i}$ for any $0<\lambda<1$ and this process would have produced another Cantor set with identical topological properties but potentially different metric ones.

We can go further and chose arbitrary diameters for $J_{n+1,i}$ for all $n\geq 0$, $1\leq i\leq 2^{n}$ provided that $0<\diam(J_{n+1,i})<\diam(I_{n,i})$, and still obtain a Cantor set.

In the following, we will use the $1/3$-Cantor set as an example of a locally self-similar space, and define a generalised Cantor set, using this notation, as a non-example of local self-similarity. The key to constructing the non-example of Theorem~\ref{thmdimtopwontdo} will be to have, as $n$ increases, the diameters of the gaps become progressively smaller proportions of the intervals they are cut from. 
\section{Product of a Cantor set and a hyper-cube}\label{sectionCtimesI}

As discussed in the introduction, the H\"older dimension of a space can be not attained, and we gave the $1/3$-Cantor set as an example when the H\"older dimension was equal to $0$. In this section, we provide a family of examples of compact, locally self-similar spaces with H\"older dimension $n$, for any $n\in\mathbb{N}$, none of which attain their H\"older dimension. Recall,

\begin{thm}[Theorem~\ref{thmCtimesI}]\label{propcantorcrosshypercube}
	Let $n\in\mathbb{N}$, $I^{n}=[0,1]^{n}$ be the unit hyper-cube in $\mathbb{R}^{n}$, $C$ be the $1/3$-Cantor set, and $X=C\times I^{n}$ their product with the $\ell^{2}$ metric. Let $Y$ be a H\"older equivalent metric space to $X$. Then $Y$ has Hausdorff dimension strictly greater than $n$. In particular, $C\times I^{n}$ has H\"older dimension $n$ but no H\"older equivalent space attains $n$ as its Hausdorff dimension.
\end{thm}

We draw inspiration from an important method in the study of conformal dimension. In the $n=1$ case, this method tells us that the Hausdorff dimension of $C\times I$ cannot be lowered by quasi-symmetric equivalence and, therefore, $C\times I$ is minimal for conformal dimension, see~\cite{PansuConforme}. Under H\"older equivalence we don't get something quite as strong, but from this method we can still derive that, for $C\times I$, the Hausdorff dimension of H\"older equivalent spaces can never be lowered to $1$ (or below). The core idea in the quasi-symmetric case can be found in~\cite[Lemma 1.6]{bourdonpolyedreshyperboliques}; where Bourdon makes explicit ideas formulated by Pansu in~\cite[Proposition 2.9]{PansuConforme} and~\cite[Lemma 6.3]{pansuconformaldim}. Essentially, if one has a large family of curves in a space, $X$, which are spread out enough so that the Hausdorff dimension of $X$ is greater that $1$, and if these curves are still sufficiently spread out in a quantifiable way after applying, in our case, a bi-H\"older map, then the image of these curves still force the Hausdorff dimension of the image to be greater than $1$. To then generalise this to $n\geq 2$ in the quasi-symmetric case, one considers $C\times I^{n}$ as $\left(C\times I^{n-1}\right)\times I$ and uses the same idea about curves. However, this doesn't work for H\"older equivalence, so we introduce some tools to prove that we can, instead, consider a large family of copies of $I^{n}$ in a similar way. We will refer to copies of $I^{n}$ as \redit{hyper-curves}.

There are two key ingredients to this argument; we can find a family of hyper-curves, $\Gamma$, in $Y$ such that:
\begin{enumerate}
	\item The hyper-curves are uniformly `big'. Formally, there exists a uniform, $B>0$, lower bound away from zero such that for any hyper-curve $\gamma\in\Gamma$, the Hausdorff $n$-measure of $\gamma$ as a subspace of $Y$ is at least $B$. The importance of checking this `big'-ness property is that it should mean that every hyper-curve substantially contributes to the (at least) $n$-dimensionality of $Y$, and if we have enough of them, then we should break-out beyond dimension $n$.\\
	\item The hyper-curves are `spread out'. We will quantify this by fitting a measure on $\Gamma$ such that, uniformly, the measure of the set of hyper-curves intersecting a given subset of $Y$ is bounded above polynomially by the diameter of the subset. Such an upper bound encapsulates the `spread out' property, which can be seen by examining the opposite via the extreme of a Dirac mass on a single hyper-curve.
\end{enumerate} 

For any $f\colon C\times I^{n} \rightarrow Y$ bi-H\"older homeomorphism to an equivalent space $Y$, there is a natural choice for a family of hyper-curves in $Y$ obtained by pushing the fibres of $C$ in $C\times I^{n}$ through $f$. Formally, let $\mathcal{C}=\{\gamma_{x}\colon I^{n}\rightarrow C\times I^{n}\mid x\in C \}$ where $\gamma_{x}\colon \underline{t}\mapsto (x,\underline{t})$, and $\Gamma = \{f\circ\gamma_{x}\mid x\in C\}$. We proceed to show that elements of $\Gamma$ are `big' and `spread out' as described above.

Firstly, we state simply and precisely what we mean by curves being uniformly `big'.

\begin{lem}\label{lemcurvesarebig}
	If $\gamma\colon I^{n}\rightarrow Y$ is a $(\lambda,\alpha,\beta)$-bi-H\"older homeomorphism onto its image, then there exists a constant $B=B(\lambda,n)>0$ such that the Hausdorff $n$-measure of $\gamma(I^{n})$ is at least $B$.
\end{lem}

For any $\gamma\in\Gamma$, by definition, there exists $x\in C$ such that $\gamma =f \circ \gamma_{x}$. Note, for any $x\in C$, $\gamma_{x}$ is an isometry onto its image. Combine this with the $(\lambda,\alpha,\beta)$-bi-H\"older property of $f$ to make the observation that $\gamma\colon I^{n}\rightarrow Y$ is a $(\lambda,\alpha,\beta)$-bi-H\"older homeomorphism onto its image.

We need to introduce the following machinery before we can prove Lemma~\ref{lemcurvesarebig}.

\begin{prop}\label{propabstractFsurjective}
	Let $F\colon I^{n} \rightarrow I^{n}$ be a continuous map such that $F(\partial I^{n})\subseteq \partial I^{n}$ and the map $\restr
	{F}{\partial I^{n}}\colon\partial I^{n}\rightarrow \partial I^{n}$ induces a non-trivial endomorphism on the reduced $(n-1)$-homology group of the boundary, $\tilde{H}_{n-1}(\partial I^{n})$. Then $F$ is surjective.
\end{prop}

\begin{proof}
	First, note that the boundary map $\restr{F}{\partial I^{n}}$ is surjective. If not, then there exists a point $z_{0}\in\partial I^{n}\setminus F(\partial I^{n})$ and we can factor $\restr{F}{\partial I^{n}}$ through $\partial I^{n}\setminus\{z_{0}\}$ to find $\restr
	{F}{\partial I^{n}}\colon\partial I^{n}\rightarrow \partial I^{n}$ is equivalent to the path
	\[
	\partial I^{n} \xrightarrow{F} \partial I^{n}\setminus\{z_{0}\}\hookrightarrow \partial I^{n}.
	\]
	However, $\partial I^{n}\setminus\{z_{0}\}$ deformation retracts to a point, and, therefore, has trivial reduced $(n-1)$-homology. Hence, this path, through $\partial I^{n}\setminus \{z_{0}\}$, implies that the endomorphism induced by $\restr{F}{\partial I^{n}}$ on $\tilde{H}_{n-1}(\partial I^{n})$ is trivial, contradicting our assumption.
	
	Now, for a contradiction, suppose that there exists a point $y_{0}\in I^{n}\setminus F(I^{n})$. Note that this point must lie in the interior of $I^{n}$ as $F$ surjects onto the boundary by the above. As $y_{0}$ is interior, we have an inclusion $\iota\colon\partial I^{n} \hookrightarrow I^{n}\setminus\{y_{0}\}$ which has a retract $r\colon I^{n}\setminus\{y_{0}\} \rightarrow \partial I^{n}$ defined as follows. For any point $y\in I^{n}\setminus\{y_{0}\}$, let $r(y)$ be the point of intersection of $\partial I^{n}$ and the straight line $l_{y} \coloneqq\{t y + (1-t)y_{0}\mid t\geq 0\}$ that originates at $y_{0}$ and passes through $y$. Note, $r$ restricts to the identity on the boundary, so $r\circ \iota$ is the identity map on $\partial I^{n}$. The composition $\iota\circ r$ is homotopic to the identity by the straight line homotopy $(y,t)\mapsto ty+(1-t)r(y)$. This map takes values in $I^{n}\setminus\{y_{0}\}$ because $y_{0}, y, r(y)$ are co-linear, in that order along the line $l_{y}\subseteq I^{n}$, and the points $ty+(1-t)r(y)$, for $0\leq t\leq 1$, are contained in the segment of $l_{y}$ containing $y$ and $r(y)$, which does not contain $y_{0}$.
	
	Hence, $\iota$ is a homotopy equivalence and we may deduce that $\iota_{*}\colon \tilde{H}_{n-1}(\partial I^{n}) \rightarrow \tilde{H}_{n-1}(I^{n}\setminus\{y_{0}\})$ is an isomorphism.
	
	Observe that 
we have the following commutative diagram;
	\begin{center}
		\begin{tikzpicture}[every node/.style={midway}]
		\matrix[ column sep={3em}, row sep={3em}] at (0,0) {
			\node(11) {$I^{n}$}  ; & \node(12) {$I^{n}\setminus\{y_{0}\}$}; \\
			\node(21) {$\partial I^{n}$}; & \node (22) {$\partial I^{n}$};\\
		};
		\draw[right hook->] (21) -- (11) node[anchor=east] {};
		\draw[->] (11) -- (12) node[anchor=south] {$F$};
		\draw[right hook->] (22) -- (12) node[anchor=west] {$\iota$};
		\draw[->] (21) -- (22) node[anchor=north] {$F$};
		\end{tikzpicture}
	\end{center}
which induces the following commutative diagram in the the reduced $(n-1)$-homology;
\begin{center}
	\begin{tikzpicture}[every node/.style={midway}]
	\matrix[ column sep={3em}, row sep={3em}] at (0,0) {
		\node(11) {$0$}  ; & \node(12) {$\mathbb{Z}$}; \\
		\node(21) {$\mathbb{Z}$}; & \node (22) {$\mathbb{Z}$};\\
	};
	\draw[->] (21) -- (11) node[anchor=east] {};
	\draw[->] (11) -- (12) node[anchor=south] {};
	\draw[->] (22) -- (12) node[anchor=west] {$\iota_{*}$};
	\draw[->] (21) -- (22) node[anchor=north] {$F_{*}$};
	\end{tikzpicture}
\end{center}
	Further, as $F$ induces a non-trivial homomorphism $F_{*}\colon\tilde{H}_{n-1}(\partial I^{n}) \rightarrow \tilde{H}_{n-1}(\partial I^{n})$ and $\iota_{*}\colon \tilde{H}_{n-1}(\partial I^{n})\rightarrow \tilde{H}_{n-1}(I^{n}\setminus\{y_{0}\})$ is an isomorphism, 
the commutative diagram in the homology above gives contradiction.
\end{proof}

Recall, for any $\gamma\in\Gamma$, $\gamma\colon I^{n}\rightarrow Y$ is a $(\lambda,\alpha,\beta)$-bi-H\"older homeomorphism onto its image. To massage this set-up into one that can utilise Proposition~\ref{propabstractFsurjective} we introduce the following notation.

Let $A_{i}$ be the `axial' face of $I^{n}$ defined to be the subset of $I^{n}$ with value $0$ in the $i$-th coordinate;
\[
A_{i}\coloneqq I^{i-1}\times\{0\}\times I^{n-i}.
\]
Each axial face has an `opposite' face $O_{i}$ defined to be the subset of $I^{n}$ with value $1$ in the $i$-th coordinate;
\[
O_{i}\coloneqq I^{i-1}\times\{1\}\times I^{n-i}.
\]

The idea is that $\gamma$ transfers these faces over to $Y$ as distorted versions of themselves, which we use to build a map from $Y$ to the cube, and then compose with $\gamma$ to obtain a self-map of the cube to which we can apply Proposition~\ref{propabstractFsurjective}.

For each $i$, we define a map $\phi_{i}$ as a kind of projection in the $i$-th direction.
Define $\phi_{i}\colon Y \rightarrow \mathbb{R}_{\geq 0}$ by
\[
\phi_{i}(y) = \inf_{x\in A_{i}} d_{Y}(y,\gamma(x)),
\]
for any $y\in Y$. More concisely written;
\[
\phi_{i}(y) = d_{Y}(y,\gamma(A_{i})).
\]
Note, as $\phi_{i}$ is a distance function to a set, $\phi_{i}$ is $1$-Lipschitz and thus also continuous.

We could now take the product of these maps to build a map to $\mathbb{R}^{n}$, but it is unclear what the image of this map will look like. Instead, we make a slight adjustment to these maps to make their product simpler. Define $\psi_{i}\colon Y \rightarrow [0,1]$ by capping $\phi_{i}$ at $1/\lambda$ and then rescaling to $[0,1]$. That is,
\[
\psi_{i}(y) \coloneqq \lambda\max\{\phi_{i}(y),1/\lambda\}.
\]
Note that $\psi_{i}$ inherits $\phi_{i}$'s Lipschitz-ness, but is now $\lambda$-Lipschitz for each $i$.
Let $\Psi\colon Y\rightarrow I^{n}$ be the product of these maps; for any $y\in Y$ define
\begin{equation}\label{defnPsi}
\Psi(y)\coloneqq \left(\psi_{i}(y)\right)_{1\leq i \leq n},
\end{equation}
and note the following lemma.
\begin{lem}\label{lemPsiLipschitz}
	The product map $\Psi$ is $\lambda\sqrt{n}$-Lipschitz.
\end{lem}
The product map $\Psi$ has image in the unit hyper-cube, $I^{n}$, so we can reduce to studying continuous self-maps of the unit hyper-cube, with some desirable properties, by composing with $\gamma$. Define the continuous map $F\colon I^{n} \rightarrow I^{n}$ by 
\begin{equation}\label{defnF}
F \coloneqq \Psi\circ \gamma.
\end{equation}

A useful observation to make is that the $(n-1)$-dimensional faces are mapped to themselves under $F$. Indeed, for any $x\in A_{i}$, $\gamma(x)\in \gamma(A_{i})$ so $d_{Y}(\gamma(x),\gamma(A_{i}))=0$, and therefore, the $i$-th coordinate of $F(x)$ is $0$, which characterises being an element of $A_{i}$. For $x\in O_{i}$, $x$ is at least $1$ away from every point in $A_{i}$, and therefore, by the lower bound of the bi-H\"older inequality for $\gamma$, $d_{Y}(\gamma(x_{i}),\gamma(A_{i}))\geq 1/\lambda$. After applying the capping and rescaling, we see that the $i$-th component of $F(x)$ takes the value $1$, thus determining it as an element of $O_{i}$. This observation will allow us to use the following lemma when studying $F$.

\begin{lem}\label{lemboundarymaphomotopictoidentity}
	Let $F\colon I^{n}\rightarrow I^{n}$ be a continuous map such that, for all $i$, $F(A_{i})\subseteq A_{i}$ and $F(O_{i})\subseteq O_{i}$, then $F(\partial I^{n})\subseteq \partial I^{n}$ and $\restr{F}{\partial I^{n}}\colon \partial I^{n}\rightarrow \partial I^{n}$ is homotopic to the identity on $\partial I^{n}$.
\end{lem}
\begin{proof}
	Firstly, $\restr{F}{\partial I^{n}}$ has image in $\partial I^{n}$, because $\partial I^{n}$ is covered by the $(n-1)$-dimensional faces and the $(n-1)$-dimensional faces are all contained in $\partial I^{n}$.
	
	To see that $\restr{F}{\partial I^{n}}$ is homotopic to the identity on $\partial I^{n}$, consider the map $H\colon \partial I^{n}\times I \rightarrow \partial I^{n}$ defined as
	\[
	H(x,t)\coloneqq t F(x) + (1-t) x.
	\]
	At face value, $H(x,0)=x$ and $H(x,1)=F(x)$, and $H$ is continuous. However, it is not immediately obvious that the linear combination $t F(x) + (1-t) x$ is in $\partial I^{n}$ for all $t\in [0,1]$ and not just in $\mathbb{R}^{n}$. However, for all $i$, $F$ maps faces $A_{i}$ and $O_{i}$ to $A_{i}$ and $O_{i}$ respectively. Therefore, for any $x\in \partial I^{n}$, $x$ lies in a face $S$, and thus, $F(x)$ also lies in $S$. The faces of $I^{n}$ are convex, meaning that the straight line joining $x$ and $F(x)$, namely $\{t F(x) + (1-t) x\mid t\in[0,1]\}$, also lies in the face $S$. Thus we can conclude that $F$ restricted to the boundary is indeed homotopic to the identity.
\end{proof}

We can now combine Lemma~\ref{lemboundarymaphomotopictoidentity} and Proposition~\ref{propabstractFsurjective} to get the following.

\begin{lem}\label{lemselfcubemapsurjective}
	Let $F\colon I^{n}\rightarrow I^{n}$ be a continuous map such that, for all $i$, $F(A_{i})\subseteq A_{i}$ and $F(O_{i})\subseteq O_{i}$. Then $F$ is surjective.
\end{lem}

\begin{proof}
	By Lemma~\ref{lemboundarymaphomotopictoidentity}, $F(\partial I^{n})\subseteq \partial I^{n}$, and $\restr{F}{\partial I^{n}}\colon \partial I^{n}\rightarrow \partial I^{n}$ is homotopic to the identity on $\partial I^{n}$ and therefore induces the identity on $\tilde{H}_{n-1}(\partial I^{n})$. However, $\partial I^{n}$ is homotopic to $\mathbb{S}^{n-1}$ which has $\tilde{H}_{n-1}(\mathbb{S}^{n-1}) \cong \mathbb{Z}$ which is non-trivial. Hence, the endomorphism of $\tilde{H}_{n-1}(\partial I^{n})$ induced by $\restr{F}{\partial I^{n}}\colon \partial I^{n}\rightarrow \partial I^{n}$ 
	must be 
	 non-trivial.
	
	We have verified that $F$ satisfies the conditions to apply Proposition~\ref{propabstractFsurjective} allowing us to conclude that $F$ is surjective.
\end{proof}

We now have the requisite tools to prove Lemma~\ref{lemcurvesarebig}.

\begin{proof}[Proof of Lemma~\ref{lemcurvesarebig}]
	From the discussion earlier, we have an induced continuous map $F$, see~\eqref{defnF}, which maps faces to faces, and is, therefore, surjective by Lemma~\ref{lemselfcubemapsurjective}. This forces the product of projections $\Psi$, see~\eqref{defnPsi}, to be surjective as well.
	Thus, if $\mathcal{H}^{n}$ denotes the $n$-dimensional Hausdorff measure, we observe
	\[
	\mathcal{H}^{n}(I^{n})=\mathcal{H}^{n}(\Psi(Y))\leq \left(\lambda\sqrt{n}\right)^{n} \mathcal{H}^{n}(Y),
	\]
	as $\Psi$ is $(\lambda\sqrt{n})$-Lipschitz by lemma~\ref{lemPsiLipschitz}. This allows us to conclude that the Hausdorff $n$-measure of $Y$ is at least $\mathcal{H}^{n}(I^{n})/(\lambda\sqrt{n})^{n}$, which is strictly positive because, on $\mathbb{R}^{n}$, $\mathcal{H}^{n}$ is proportional to Lebesgue measure, see~\cite[Page 56]{Mattila}.
\end{proof}

Now, onto the `spread out' property;
Let $\mu_{C}$ be the probability measure on $C$ defined as the weak* limit of measures on the covers $\bigcup_{j=1}^{2^{i}}I_{i,j}$ defined by letting the measure of each interval $I_{i,j}$ be $1/2^{i}$. 
This induces a measure, $\mu_{\Gamma}$, on $\Gamma$ by
\begin{equation}\label{defnmugamma}
\mu_{\Gamma}(S) \coloneqq \mu_{C}(\{x\in C\mid f\circ \gamma_{x} \in S\}),
\end{equation}
for any $S \subseteq \Gamma$. We claim $\mu_{\Gamma}$ has a `spread out' property in the form of the following lemma.

\begin{lem}\label{lemcurvesarespread}
	There exists $A>0$ such that, for any $U\subseteq Y$,
	\[
	\mu_{\Gamma}\left(\{\gamma\in\Gamma\mid\gamma\cap U\neq \emptyset \}\right)\leq A\diam(U)^{\log{2}/\alpha\log{3}}.
	\]
\end{lem}

\begin{proof}
	We start by noting a similar upper bound for $\mu_{C}$. The measure $\mu_{C}$ on $C$ is ``Ahlfors $\log(2)/\log(3)$-regular'', see~\cite[Theorem 1.14]{Falconergeometryoffractalsets}, in particular, there exists $\nu>0$ such that for any closed ball $\overbar{B}$ in $C$,
	\begin{equation}\label{eqCmassdistributed}
	\mu_{C}\left(\overbar{B}\right)\leq \nu\diam\left(\overbar{B}\right)^{\log(2)/\log(3)}.
	\end{equation}

		Let $\phi\colon C\times I^{n}\rightarrow C$, defined by $\left(x,\underline{t}\right)\mapsto x$, be the projection of $C\times I^{n}$ onto $C$, and observe that $\phi$ is $1$-Lipschitz. 
	Now, for any subset $U\subseteq Y$ pick $y\in U$; we can assume $U$ is non-empty as the inequality trivially holds for empty $U$. Let $\mathcal{I}_{U}\coloneqq\{x\in C\mid f(\{x\}\times I^{n})\cap U \neq \emptyset \}\subseteq C$, and note that 
	\begin{equation}\label{eqSU}
	\mathcal{I}_{U} =\phi(f^{-1}(U)).
	\end{equation}
	The $1$-Lipschitz property of $\phi$ combined with the bi-H\"older inequalities for $f$ tells us that
	\begin{equation}\label{eqphifisbiholder}
	\diam(\mathcal{I}_{U})\leq\diam(f^{-1}(U))\leq \left(\lambda\diam(U)\right)^{1/\alpha}.
	\end{equation}
	Further, if $\overbar{B}$ is the closed ball in $C$ of radius $\diam(\mathcal{I}_{U})$ centred at $\phi(f^{-1}(y))$, then $\mathcal{I}_{U}\subseteq\overbar{B}$ and
	\begin{equation}\label{eqclosedballapprox}
	\diam\left(\overbar{B}\right)\leq 2\diam(\mathcal{I}_{U}).
	\end{equation}
	Therefore, we can conclude
	\begin{align*}
	\mu_{\Gamma}\left(\{\gamma\in\Gamma\mid\gamma\cap U\neq \emptyset \}\right)&=\mu_{C}(\mathcal{I}_{U}) &&\text{by~\eqref{defnmugamma} and~\eqref{eqSU},}\\
	&\leq \mu_{C}\left(\overbar{B}\right) &&\text{as $\mathcal{I}_{U}\subseteq B$,}\\
	&\leq\nu\diam\left(\overbar{B}\right)^{\log 2/\log 3} &&\text{by~\eqref{eqCmassdistributed},}\\
	&\leq 2^{\log(2)/\log(3)}\nu \diam(\mathcal{I}_{U}) &&\text{by~\eqref{eqclosedballapprox},}\\
	&\leq A\diam(U)^{\log{2}/\alpha\log{3}} &&\text{by~\eqref{eqphifisbiholder},}
	\end{align*}
	where $A= 2^{\log{2}/\log{3}}\nu\lambda^{\log{2}/\alpha\log{3}}$. 
\end{proof}
We now know something about arbitrary decompositions of each $\gamma\in\Gamma$ via Lemma~\ref{lemcurvesarebig}, and something about how decompositions of $Y$ interact with $\Gamma$ via Lemma~\ref{lemcurvesarespread}. We introduce the following notation for indicator functions, as they will be useful for converting decompositions of $Y$ to decompositions for $\gamma\in\Gamma$, which is how we shall link these two ideas. For any $U\subseteq Y$ and $\gamma\in\Gamma$, define
\[
\mathbbm{1}_{U}(\gamma) = \begin{cases}
1 &\text{ if }\gamma\cap U\neq \emptyset,\\
0 &\text{ otherwise}.
\end{cases}
\]

We now have sufficient tools to prove Theorem~\ref{propcantorcrosshypercube}.

\begin{proof}[Proof of Theorem~\ref{propcantorcrosshypercube}]
	Let $B>0$ be as in Lemma~\ref{lemcurvesarebig}.
	Observe, as $\mu_{\Gamma}$ is a probability measure,
	\begin{align*}
	B &= \int_{\Gamma} B \diff \mu_{\Gamma}(\gamma), \\
	\intertext{so, for any decomposition $\mathcal{U}$ of $Y$, by Lemma~\ref{lemcurvesarebig},}
	&\leq \int_{\Gamma} \sum_{U\in\mathcal{U}}\diam(U)^{n} \mathbbm{1}_{U}(\gamma) \diff \mu_{\Gamma}(\gamma),\\
	&= \sum_{U\in\mathcal{U}} \diam(U)^{n}\int_{\Gamma} \mathbbm{1}_{U}(\gamma)\diff \mu_{\Gamma}(\gamma),\\
	\intertext{then, by Lemma~\ref{lemcurvesarespread}, there exists some $A>0$ such that,}
	&\leq \sum_{U\in\mathcal{U}}\diam(U)^{n} A\diam(U)^{\log{2}/\alpha\log{3}},\\
	&= A\sum_{U\in\mathcal{U}} \diam(U)^{n+\log 2/\alpha\log 3}.
	\end{align*}
	Hence,
	\[
	\sum_{U\in\mathcal{U}} \diam(U)^{n+\log 2/\alpha\log 3} \geq \frac{B}{ A}>0,
	\]
	and therefore, the Hausdorff dimension of $Y$ is at least $n+\log2/\alpha\log 3>n$ as $\mathcal{U}$ was arbitrary.

To see that the H\"older dimension of $C\times I^{n}$ is equal to $n$, we use Corollary~\ref{cor1}. The compactness of  $C\times I^{n}$ comes from being a product of compact spaces. Considering $C\times I^{n}$ as a subspace of $\mathbb{R}^{n+1}$, we see that it is doubling directly from the doubling property of $\mathbb{R}^{n+1}$. Finally, $C\times I^{n}$ has capacity dimension $n$ by Proposition~\ref{propcapdimCtimesIis1} below.
\end{proof}
We now present a proof of the capacity dimension result for $C\times I^{n}$ used above.
\begin{prop}\label{propcapdimCtimesIis1}
	$C\times I^{n}$ has capacity dimension $n$.
\end{prop}

For this proposition, we utilise the following lemma, which is Theorem 9.5.1 in~\cite{Elementsofasymptoticgeometry}. The interested reader should note that Buyalo and Schroeder refer to capacity dimension as ``$\ell$-dimension'' in this source.
\begin{lem}\label{lemcapdimofproduct}
	For any metric spaces $X_{1}$ and $X_{2}$, the capacity dimension of $X_{1}\times X_{2}$ is at most the sum of the capacity dimensions of $X_{1}$ and $X_{2}$.
\end{lem}
\begin{proof}[Proof of Proposition~\ref{propcapdimCtimesIis1}]
	Note that $C\times I^{n}$ contains a copy of $I^{n}$ as $\{0\}\times I^{n}$, and therefore has topological dimension at least $n$. As topological dimension is a lower bound to capacity dimension, we observe that $C\times I^{n}$ also has capacity dimension at least $n$.
	Therefore, using Lemma~\ref{lemcapdimofproduct}, it only remains to check that the capacity dimensions of $C$ and $I^{n}$ are at most $0$ and $n$ respectively.
	
	For $C$, take any $0<\delta\leq 1$ and let $n\in\mathbb{N}$ such that $1/3^{n}\leq\delta<1/3^{n-1}$. The cover $\{I_{n,i}\cap C\mid 1\leq i\leq 2^{n}\}$ has mesh at most $1/3^{n}\leq \delta$, multiplicity $1$, and Lebesgue number at least $1/3^{n}> \delta/3$. Therefore, $C$ has capacity dimension at most $0$.
	
	For $I^{n}$, we show that $I$ has capacity dimension at most $1$, then inductively use Lemma~\ref{lemcapdimofproduct} to prove that $I^{n}$ has capacity dimension at most $n$.
	
	For $I$, take any $0<\delta\leq 1$. The cover of $I$ by balls of radius $\delta/2$ centred at $n\delta/2$, for $n\in \mathbb{N}$ and $0\leq n\leq (2/\delta)+1$,
	has mesh at most $\delta$, multiplicity $2$, and Lebesgue number at least $\delta/4$. Therefore, $I$ has capacity dimension at most $1$.
\end{proof}

\section{Capacity dimension versus topological dimension}\label{sectioncapvstop}

Theorem~\ref{thmmain} shows that the H\"older dimension of a compact, doubling space is at most its capacity dimension. However, as topological dimension is a more commonly used notion of dimension, one could ask if H\"older dimension is, in fact, at most the space's topological dimension, extending the self-similar case. In this section, we provide an example of a compact, doubling space which has topological dimension $0$ but H\"older dimension $1$, proving that ``capacity dimension'' cannot be replaced with ``topological dimension'' in Theorem~\ref{thmmain}.

\begin{thm}[Theorem~\ref{thmdimtopwontdo}]
	Let $X$ be the Cantor set defined in Section~\ref{sectionnotationforcantorsets} where the diameter of the gaps, $\diam(J_{n,i})$, is taken to be $\frac{1}{10n^{n}}$ for all $n\geq 1$ and $1\leq i\leq 2^{n-1}$. Then $X$ has H\"older dimension equal to $1$.
\end{thm}

As $X$ is a Cantor set, it has topological dimension $0$ and is compact. It is doubling as it is a subspace of $\mathbb{R}$ which is doubling. Therefore, to accomplish the goal stated above, we need only prove this theorem.

The main idea is, by making the gaps shrink fast enough that they cannot account for all the Hausdorff $1$-measure in $I$, we've forced $X$ to have Hausdorff dimension $1$.
Furthermore, as the shrinking is faster than any fixed power of $n$, no H\"older equivalence can find an equivalent space $Y$ without this `fast-shrinking gap' property, meaning any equivalent space will also have Hausdorff dimension $1$.

Our construction of $X$ allows us to choose the diameters for the gaps, $J_{n,i}$, but leaves the diameters of $I_{n,i}$ implicit. Investigation of the construction gives us the following easy, but useful, bound.

\begin{lem}\label{lemintervalsbetweenthirdandhalf}
	\[
	\frac{1}{3^{n}}\leq\diam(I_{n,i}),
	\]
	for every $n\in\mathbb{N}$ and $1\leq i\leq 2^{n}$.
\end{lem}

\begin{proof}
	Let $C$ be the $1/3$-Cantor set as constructed in Section~\ref{sectionnotationforcantorsets}, with intervals $I^{C}_{n,i}$ and gaps $J^{C}_{n+1,i}$. Recall that $\diam(J^{C}_{n+1,i})=1/3^{n+1}$ for all $n\geq 0$ and $1\leq i \leq 2^{n}$. Compare this with how we defined the gaps in $X$ to have diameter $1/10(n+1)^{n+1}$ to see that we cut out at most the middle third of every interval in the construction of $X$;
	\[
	\frac{1}{10(n+1)^{n+1}}\leq \frac{1}{3^{n+1}},
	\]
	for all $n\geq 0$. Therefore, inductively we see that $\diam(I_{n,i})\geq \diam(I^{C}_{n,i})=1/3^{n}$.
\end{proof}

It is helpful to note the following lemma, which is a result of exclusively cutting from the interior of intervals in the construction of $X$.

\begin{lem}\label{lemenpointsarein}
	Endpoints of intervals $I_{n,i}$ lie in $X$.
\end{lem}

The following is~\cite[Corollary 1]{McShane}.
\begin{lem}\label{lemmcshane}
	Let $S$ be a subset of a metric space $Z$, and let $g\colon S\rightarrow \mathbb{R}$ be a real-valued H\"older continuous function.
	Then $g$ can be extended to $Z$ preserving the H\"older condition.
\end{lem}

We now have sufficient tools to prove the theorem.

\begin{proof}[Proof of Theorem~\ref{thmdimtopwontdo}.]
Note that 
\[
[0,1]\setminus X = \bigcup_{k,j}J_{k,j},
\]
so, for any countable decomposition of $X$, say $\mathcal{A}=\{A_{i}\}_{i\in\mathbb{N}}$, we can extend $\mathcal{A}$ to a decomposition of $[0,1]$ by including all the $J_{k,j}$. Also, using that the Hausdorff $1$-measure of $[0,1]$ is $1$,
\[
\sum_{i} \diam(A_{i}) + \sum_{k,j} \diam(J_{k,j}) \geq \mathcal{H}^{1}([0,1]) = 1,
\]
but
\[
\sum_{k,j} \diam(J_{k,j}) = \sum_{k=1}^{\infty} \sum_{j=1}^{2^{k-1}} \frac{1}{10k^{k}} = \frac{1}{10}\sum_{k=1}^{\infty} \frac{2^{k}}{k^{k}}.  
\]
For $k\geq 4$, $k^{k}\geq 4^{k}= 2^{2k}$ so we see
\[
\sum_{k=1}^{\infty} \frac{2^{k}}{k^{k}} \leq \sum_{k=1}^{3} \frac{2^{k}}{k^{k}} + \sum_{k=4}^{\infty} \frac{2^{k}}{2^{2k}} \leq \sum_{k=1}^{3} \frac{2^{k}}{k^{k}} +\sum_{k=1}^{\infty} \frac{1}{2^{k}} \leq \sum_{k=1}^{3} \frac{2^{k}}{k^{k}} +1 < 5 < \infty.
\]
Hence, $\sum_{k,j} \diam(J_{k,j})<5/10 = 1/2$. Therefore, $\sum_{i} \diam(A_{i})>1/2$. The decomposition $A$ was arbitrary so $\mathcal{H}^{1}(X)\geq 1/2>0$ and hence $\dim_{H}(X) \geq 1$, but $\dim_{H} (X) \leq \dim_{H}([0,1]) = 1$, so together we get that $\dim_{H} = 1$.

Now consider $f\colon X\rightarrow Y$, a $(\lambda,\alpha,\beta)$ bi-H\"older homeomorphism between $X$ and a metric space $Y$. We would like to prove that $\dim_{H} (Y) \geq 1$. To do this, let's reduce to working in $\mathbb{R}$ so that we can extend decompositions to decompositions of intervals, like in the above. Consider $\psi\colon Y \rightarrow \mathbb{R}$ defined by $y\mapsto d_{Y}(f(0),y)$. Note that $\psi$ is a 1-Lipschitz map via the triangle inequality. Hence, we can preserve the upper bound of our H\"older inequality for $f$ when we compose with $\psi$. That is, for any $z_{1},z_{2}\in X$
\begin{equation}\label{equationbadcantorholdercont}
d_{\mathbb{R}}(\psi(f(z_{1})),\psi(f(z_{2}))) \leq d_{Y}(f(z_{1}),f(z_{2}))\leq \lambda d_{X} (z_{1},z_{2})^{\beta}.
\end{equation}

Hence, by Lemma~\ref{lemmcshane} there exists an extension, $F$, of $\psi\circ f$ to $I$ that is $(\lambda ,\beta)$-H\"older continuous too. This extension means we can derive information from the gaps, $J_{k,j}$, too, instead of just from the space $X$. For instance, for any $j,k$,
\[
\diam(J_{k,j}) = \frac{1}{10k^{k}} \implies \diam(F(J_{k,j})) \leq \lambda \left(\frac{1}{10k^{k}}\right)^{\beta}.
\]
We interpret this as the H\"older map, $F$, being unable to break the shrinking property of the gaps.

Consider $I_{k,1}$, which has width at least $1/3^{k}$ by Lemma~\ref{lemintervalsbetweenthirdandhalf}. We know that the smaller endpoint of $I_{k,1}$ is $0$, and let its larger endpoint be $x$, for some $x>0$. By Lemma~\ref{lemenpointsarein}, both $0$ and $x$ lie in $X$ and, therefore, $F$ evaluates to $\psi\circ f$ on them as $F$ is an extension of $\psi\circ f$. Hence, 
\begin{align*}
|F(0)-F(x)|&=|\psi(f(0))-\psi(f(x))|,\\
&= |0-d_{Y}(f(0),f(x))|,\\
&= d_{Y}(f(0),f(x)) \geq \frac{1}{\lambda} d_{X}(0,x)^{\alpha}\geq  \frac{1}{\lambda 3^{\alpha k}}.
\end{align*}
Thus, $F(I_{k,1})$ contains $\left[0,1/(\lambda 3^{\alpha k})\right]$ as a subset, by the Intermediate Value Theorem.

Now, the gaps within $I_{k,1}$ are precisely $J_{i,j}$ where $i\geq k+1$ and $1\leq j\leq 2^{i-(k+1)}$. Each $J_{i,j}$ has diameter $1/10i^{i}$, so the corresponding gap, $F(J_{i,j})$, in the image has diameter at most $\lambda \left(1/(10{i+1}^{i+1})\right)^{\beta}$, by the H\"older continuity of $F$.
Thus, the 1-measure of gaps in the image is controlled as follows;
\begin{align}
\sum_{i,j} \diam(F(J_{i,j})) &\leq \sum_{n=k+1}^{\infty} 2^{n-(k+1)}\lambda \left(\frac{1}{10n^{n}}\right)^{\beta},\nonumber\\
 &\leq \frac{\lambda}{2^{k+1}} \left(\frac{1}{10}\right)^{\beta} \sum_{n=k+1}^{\infty} \left(\frac{2}{(k+1)^{\beta}}\right)^{n},\nonumber \\
 &= \frac{\lambda}{10^{\beta}2^{k+1}} \left(\frac{2}{(k+1)^{\beta}}\right)^{k+1} \frac{(k+1)^{\beta}}{(k+1)^{\beta}-2},\nonumber\\
 &= \frac{\lambda }{10^{\beta}(k+1)^{\beta k}((k+1)^{\beta}-2)}, \nonumber\\
 &\leq \frac{\lambda }{10^{\beta}(k+1)^{\beta k}}.\label{equationsumofgaps1}
\end{align}

The last inequality holds for $(k+1)^{\beta}\geq 3$, which will be true for sufficiently large $k$, because $(k+1)^{\beta} \to \infty$ as $k\to \infty$. Note,
\[
\left(\frac{(k+1)^{\beta}}{3^{\alpha}}\right)^{k}\to\infty,\text{ as }k\to\infty.
\]
Hence, by taking $k$ sufficiently large, we may assume simultaneously;
\begin{itemize}
	\item $(k+1)^{\beta}\geq 3$, and\\
	\item $\left((k+1)^{\beta}/3^{\alpha}\right)^{k}>2\lambda^{2}/10^{\beta}$.
\end{itemize}

The latter is important because it is equivalent to
\begin{equation}\label{equationsumofgaps2}
\frac{\lambda }{10^{\beta}(k+1)^{\beta k}} <\frac{1}{2} \frac{1}{\lambda  3^{\alpha k}},
\end{equation}
which allows us to conclude that gaps cannot account for all the $1$-measure in $Y$. More precisely, for any countable  decomposition, $B=\{B_{j}\}_{j\in\mathbb{N}}$, of $Y$, define $A_{j} = f^{-1}(B_{j})$ for all $j$, and let $A=\{A_{j}\}_{j\in\mathbb{N}}$ be the decomposition of $X$ induced by pulling $B$ back through $f$. Note that $A$  is also a cover for $I_{k,1}\cap X$, and if we add in the gaps contained in $I_{k,1}$, then we have a cover for $I_{k,1}$. Explicitly, $A^{\prime}\coloneqq A \cup \{J_{i,j}\mid i\geq k+1, 1\leq j \leq 2^{i-(k+1)} \}$ covers $I_{k,1}$. Hence, $F(A^{\prime})$ covers $F(I_{k,1})$, and therefore
\[
\sum_{j} \diam (F(A_{j})) +\sum_{i,j} \diam (F(J_{i,j})) \geq \mathcal{H}^{1} (F(I_{k,1})) \geq \frac{1}{\lambda  3^{\alpha k}}.
\]
From equations~\eqref{equationsumofgaps1} and~\eqref{equationsumofgaps2}, for sufficiently large $k=k(\alpha,\beta,\lambda)$ independent of the decomposition $B$, 
\[
\sum_{J_{i,j}\subset I_{k,1}} \diam (F(J_{i,j})) < \frac{1}{2\lambda  3^{\alpha k}}.
\]
Hence,
\[
\sum_{j} \diam (F(A_{j})) \geq \frac{1}{2\lambda  3^{\alpha k}}>0,
\]
for all decompositions $B$. Now, by definition and as $A_{j}\subseteq X$ for all $j$, $F(A_{j}) = \psi \circ f(A_{j}) = \psi (B_{j})$, and $\diam (\psi(B_{j})) \leq \diam (B_{j})$, so
\[
\sum_{j} \diam (B_{j}) \geq \sum_{j} \diam (F(A_{j})) \geq \frac{1}{2\lambda  3^{\alpha k}}.
\]

Hence, $\mathcal{H}^{1}(Y) \geq 1/2\lambda  3^{\alpha k}$ and $\dim_{H} (Y) \geq 1$.
\end{proof}

\section{H\"older dimension can be less than capacity dimension}\label{sectionholdimnotalwayscap}

In this section, we give an example to illustrate that the inequality between capacity dimension and H\"older dimension in Theorem~\ref{thmmain} cannot be upgraded to an equality. That is,

\begin{thm}
	Let $X=\{0\}\cup \{1/n\mid n\in\mathbb{N}\}\subset \mathbb{R}$, then $X$ is a compact, doubling metric space with capacity dimension $1$, but has H\"older dimension $0$.
\end{thm}

\begin{proof}
	Note that $X$ is compact because it is a closed and bounded subspace of $\mathbb{R}$, and $X$ is doubling because it is a subspace of $\mathbb{R}$ which is doubling.
	
	We can easily verify that $X$ has H\"older dimension $0$. Indeed, as countable collections of points have Hausdorff dimension $0$, $X$ has Hausdorff dimension 0. Further, H\"older dimension is non-negative and the Hausdorff dimension of $X$ is an upper bound for its H\"older dimension, so $X$ has H\"older dimension $0$. 
	
	To see that $X$ has capacity dimension $1$, observe that $X$ has capacity dimension at most $1$ as it is a subspace of $\mathbb{R}$, which has capacity dimension $1$. We now prove $X$ has capacity dimension at least $1$ by proving that it does not have capacity dimension $0$.
	
	For a contradiction, assume that $X$ has capacity dimension $0$ with coefficient $\sigma$. Take $n\in\mathbb{N}$ such that $n>\max\{2/\sigma,2\}$, and so that $2/(\sigma n(n-1))>0$ is sufficiently small as to apply the definition of capacity dimension $0$ with $\delta \coloneqq 2/(\sigma n(n-1))$. Let $\mathcal{U}$ be an open cover of $X$ of mesh at most $\delta$, Lebesgue number at least $\sigma\delta$, and multiplicity $1$ as per $X$ having capacity dimension $0$ with coefficient $\sigma$. 
	For any $m\geq n$, we calculate
	\[
	d\left(\frac{1}{m},\frac{1}{m-1}\right) = \frac{m-(m-1)}{m(m-1)} =\frac{1}{m(m-1)}\leq \frac{1}{n(n-1)}.
	\]
	By observing that $1/n(n-1)<2/(n(n-1)) =\sigma\delta$, we see that, for every $m\geq n$, $\{1/m,1/(m-1)\}\subseteq U_{m}$, for some $U_{m}\in \mathcal{U}$, by the Lebesgue number property. From the multiplicity $1$ restriction on $\mathcal{U}$, for any $U,V\in\mathcal{U}$, if $U\cap V \neq \emptyset$, then $U=V$. Hence, $U_{m}=U_{m+1}$ for every $m\geq n$ as $m\in U_{m} \cap U_{m+1}$. Inductively, $m\in U_{n}$ for every $m\geq n-1$. Also, $0\in U_{n}$, as $n(n-1)\geq n-1$ and $d(0,1/(n(n-1)))<\sigma \delta$, so there exists $U_{0}$ in $\mathcal{U}$ containing $0$ and $1/(n(n-1))$, but $1/(n(n-1))\in U_{n}$ so $U_{0}=U_{n}$ too. Therefore, $\diam(U_{n})\geq d(0,1/(n-1))=1/(n-1)>2/(\sigma n(n-1)) =\delta $ as $2/(\sigma n)<1$, by definition of $n$. This contradicts the mesh constraint on $\mathcal{U}$. Therefore, no such $\mathcal{U}$ exists and $X$ cannot have capacity dimension $0$.
\end{proof}


\bibliographystyle{alpha}
\bibliography{References}

\begin{thebibliography}{Pan89b}

\bibitem[Ako06]{Hakobyan}
G.~A. Akopyan.
\newblock Cantor sets that are minimal for quasi-symmetric mappings.
\newblock {\em Izv. Nats. Akad. Nauk Armenii Mat.}, 41(2):5--13, 2006.

\bibitem[BL07]{BuyaloLebadeva}
S.~V. Buyalo and N.~D. Lebedeva.
\newblock Dimensions of locally and asymptotically self-similar spaces.
\newblock {\em Algebra i Analiz}, 19(1):60--92, 2007.

\bibitem[Bou95]{bourdonpolyedreshyperboliques}
M.~Bourdon.
\newblock Au bord de certains poly\`edres hyperboliques.
\newblock {\em Ann. Inst. Fourier (Grenoble)}, 45(1):119--141, 1995.

\bibitem[BS07]{Elementsofasymptoticgeometry}
S.~Buyalo and V.~Schroeder.
\newblock {\em Elements of asymptotic geometry}.
\newblock EMS Monographs in Mathematics. European Mathematical Society (EMS),
  Z\"{u}rich, 2007.

\bibitem[Fal86]{Falconergeometryoffractalsets}
K.~J. Falconer.
\newblock {\em The geometry of fractal sets}, volume~85 of {\em Cambridge
  Tracts in Mathematics}.
\newblock Cambridge University Press, Cambridge, 1986.

\bibitem[Hei01]{heinonen2001lectures}
J.~Heinonen.
\newblock {\em Lectures on Analysis on Metric Spaces}.
\newblock Hochschultext / Universitext. Springer New York, 2001.

\bibitem[HW48]{hurewicz2015dimension}
W.~Hurewicz and H.~Wallman.
\newblock {\em Dimension Theory}.
\newblock Princeton Mathematical Series. Princeton University Press, 1948.

\bibitem[Mat95]{Mattila}
P.~Mattila.
\newblock {\em Geometry of sets and measures in {E}uclidean spaces}, volume~44
  of {\em Cambridge Studies in Advanced Mathematics}.
\newblock Cambridge University Press, Cambridge, 1995.
\newblock Fractals and rectifiability.

\bibitem[McS34]{McShane}
E.~J. McShane.
\newblock Extension of range of functions.
\newblock {\em Bull. Amer. Math. Soc.}, 40(12):837--842, 1934.

\bibitem[MT10]{confdim}
J.~M. Mackay and J.~T. Tyson.
\newblock {\em Conformal dimension}, volume~54 of {\em University Lecture
  Series}.
\newblock American Mathematical Society, Providence, RI, 2010.
\newblock Theory and application.

\bibitem[Mun00]{munkrestopology}
J.~R. Munkres.
\newblock {\em Topology}.
\newblock Featured Titles for Topology Series. Prentice Hall, Incorporated,
  2000.

\bibitem[Pan89a]{PansuConforme}
P.~Pansu.
\newblock Dimension conforme et sph\`ere \`a l'infini des vari\'{e}t\'{e}s \`a
  courbure n\'{e}gative.
\newblock {\em Ann. Acad. Sci. Fenn. Ser. A I Math.}, 14(2):177--212, 1989.

\bibitem[Pan89b]{pansuconformaldim}
P.~Pansu.
\newblock M\'etriques de {C}arnot-{C}arath\'eodory et quasiisom\'etries des
  espaces sym\'etriques de rang un.
\newblock {\em Ann. of Math. (2)}, 129(1):1--60, 1989.

\bibitem[Szp37]{Szpilrajn}
E.~Szpilrajn.
\newblock {La dimension et la mesure.}
\newblock {\em {Fund. Math.}}, 28:81--89, 1937.

\end{thebibliography}

\Addresses
\end{document}